\documentclass[11pt]{amsart}
\usepackage{amsfonts}
\usepackage{amsmath}
\usepackage{amssymb}
\usepackage{color}
\usepackage{axodraw}
\usepackage{graphicx}
\usepackage{xspace}
\usepackage{stmaryrd}
\usepackage{shuffle}
 \font \eightrm=cmr8

 \newcommand{\nc}{\newcommand}

\nc{\surj}{\to\hskip -3mm \to}

\newtheorem{thm}{Theorem}
\newtheorem{exam}[thm]{Example}

\newtheorem{cor}[thm]{Corollary}

\newtheorem{lem}[thm]{Lemma}
\newtheorem{prop}[thm]{Proposition}
\newtheorem{defn}[thm]{Definition}
\newtheorem{rmk}[thm]{Remark}

\nc{\ignore}[1]{{}}
\nc{\mrm}[1]{{\rm #1}}
\nc{\dirlim}{\displaystyle{\lim_{\longrightarrow}}\,}
\nc{\invlim}{\displaystyle{\lim_{\longleftarrow}}\,}
\nc{\vep}{\varepsilon} \nc{\ep}{\epsilon}
\nc{\sigmat}{\widetilde\sigma}
\nc{\ostar}{\overline{*}}

\nc{\mchar}{\mrm{Char}}
\nc{\Hom}{\mrm{Hom}}
\nc{\id}{\mrm{id}}

\nc{\remark}{\noindent{\bf{Remark:}}}
\nc{\remarks}{\noindent{\bf{Remarks:}}}

 \nc{\delete}[1]{}
 \nc{\grad}[1]{^{({#1})}}
 \nc{\fil}[1]{_{#1}}

\newcommand{\harpoonl}{\overset{\leftharpoonup}}
\newcommand{\harpoonr}{\overset{\rightharpoonup}}

\nc{\dd}{\mathrm{d}}
\nc{\BA}{{\mathbb A}} \nc{\CC}{{\mathbb C}} \nc{\DD}{{\mathbb D}}
\nc{\EE}{{\mathbb E}} \nc{\FF}{{\mathbb F}} \nc{\GG}{{\mathbb G}}
\nc{\HH}{{\mathbb H}} \nc{\LL}{{\mathbb L}} \nc{\NN}{{\mathbb N}}
\nc{\PP}{{\mathbb P}} \nc{\QQ}{{\mathbb Q}} \nc{\RR}{{\mathbb R}}
\nc{\TT}{{\mathbb T}} \nc{\VV}{{\mathbb V}} \nc{\ZZ}{{\mathbb Z}}
\nc{\Cal}[1]{{\mathcal {#1}}}
\nc{\mop}[1]{\mathop{\hbox {\rm #1} }\nolimits}
\nc{\smop}[1]{\mathop{\hbox {\eightrm #1} }\nolimits}
\nc{\mopl}[1]{\mathop{\hbox {\rm #1} }\limits}
\nc{\frakg}{{\frak g}}
\nc{\g}[1]{{\frak {#1}}}
\def \restr#1{\mathstrut_{\textstyle |}\raise-8pt\hbox{$\scriptstyle #1$}}
\def \srestr#1{\mathstrut_{\scriptstyle |}\hbox to
  -1.5pt{}\raise-4pt\hbox{$\scriptscriptstyle #1$}}
\nc{\wt}{\widetilde}
\nc{\wh}{\widehat}
\nc{\un}{\hbox{\bf 1}}
\nc{\redtext}[1]{\textcolor{red}{\tt #1}}
\nc{\bluetext}[1]{\textcolor{blue}{#1}}
\nc{\R}{\mathbb R}
\nc\fleche[1]{\mathop{\hbox to #1 mm{\rightarrowfill}}\limits}
\def\semi{\mathrel{\times}\kern -.85pt\joinrel\mathrel{\raise
    1.4pt\hbox{${\scriptscriptstyle |}$}}}
\nc{\np}{/\hskip -2.3mm\pi}
\nc{\snp}{/\hskip -1.8mm\pi}

\def\ta1{{\scalebox{0.2}{ 
\begin{picture}(12,12)(38,-38)
\SetWidth{0.5} \SetColor{Black} \Vertex(45,-33){5.66}
\end{picture}}}}
 
\begin{document}

\title[Rota--Baxter Algebra methods in Integral Calculus]
      {From iterated integrals and chronological\\ calculus to Hopf and Rota--Baxter algebras}

\author{Kurusch Ebrahimi-Fard}
\address{Department of Mathematical Sciences, 
		Norwegian University of Science and Technology (NTNU), 
		7491 Trondheim, Norway.}
         \email{kurusch.ebrahimi-fard@ntnu.no}         
         \urladdr{https://folk.ntnu.no/kurusche/}

\author{Fr\'ed\'eric Patras}
\address{Universit\'e C\^ote d'Azur,
               CNRS, 
               Laboratoire J.-A.~Dieudonn\'e, UMR 7351,
         		Parc Valrose,
         		06108 Nice Cedex 02, France.}
\email{patras@unice.fr}
\urladdr{www-math.unice.fr/$\sim$patras}	    

\date{}

\begin{abstract}
Gian-Carlo Rota mentioned in one of his last articles the problem of developing a theory around the notion of integration algebras, which should be dual to the one of differential algebras. This idea has been developed historically along three lines: using properties of iterated integrals and generalisations thereof, as in the theory of rough paths; using the algebraic structures underlying chronological calculus such as shuffle and pre-Lie products, as they appear in theoretical physics and control theory; and, more generally, using a particular operator identity which came to be known as Rota--Baxter relation. The recent developments along each of these lines of research and their various application domains are not always known to other communities in mathematics and related fields. The general aim of this survey is therefore to present a modern and unified perspective on these approaches, featuring among others non-commutative Rota--Baxter algebras and related Hopf and Lie algebraic structures such as descent algebras, algebras of simplices as well as shuffle, pre- and post-Lie algebras and their enveloping algebras. Accordingly, two viewpoints are proposed. A geometric one, that can be derived using combinatorial properties of Euclidean simplices and the duality between integrands and integration domains in iterated integrals. An algebraic one, provided by the axiomatisation of integral calculus in terms of Rota--Baxter~algebra.
\end{abstract}

\maketitle



\tableofcontents


\section{Introduction}
\label{sect:intro}

Chen's theory of iterated integrals \cite{Chen71} and the so-called chronological calculus (the latter being based on the use of time-ordered products of operators $M(t)$ and $N(t)$ such as $\int_0^t\dot{M}(u)N(u)du$, as for example in the seminal works of Agrachev and Gamkrelidze \cite{AG}) are classical topics in analysis, stochastic integration and control theory. They have evolved recently through the introduction of combinatorial Hopf algebra techniques such as the ones related to permutations or (free) pre-Lie algebras and their enveloping algebras. 

In the first part of the present article (sections 2 and 3), we survey these developments using a simple but not popular enough observation, namely: the fundamental duality between integration domains and integrands governing calculus with iterated integrals. This approach, originating in the combinatorial properties of simplices and its close links with the combinatorics of symmetric groups and Lie algebras \cite{patring}, was developed in \cite{AG2,bmp} to deal with operator-valued iterated integrals and in \cite{patit} in the context of Chen's approach to loop space cohomology. It allows to encode the usual algebraic free Lie and pre-Lie techniques in geometric terms. The philosophy underlying these constructions is in the vein of Chen's, in that it relates in the end the rules of classical calculus with the ones of simplicial calculus in algebraic topology, allowing to extend Chen's ideas to various non-commutative and non-associative structures.

The second part of the article (section 4) extends these ideas and techniques to a broader context, making also explicit the link (following in particular the ideas put forward in \cite{EFP2014}) between the direct approach of the first part (sections 2 and 3) and the algebra-axiomatic one in the second part. The latter can be traced back to the period 1960 --  1972, when Pierre~Cartier and Gian-Carlo~Rota, together with collaborators, set out to develop a general theory of the algebraic and combinatorial structures underlying integral calculus. Nowadays, this approach has been subsumed by what is better known as the theory of Rota--Baxter algebras. However, in spite of the fact that those algebras need not to be commutative, a large part of the principal results had first been obtained in the context of commutative function algebras. Interest in the general, i.e., non-commutative case, arose around 2000, in areas including, among others, integral and finite difference calculus for operator algebras. Perturbative quantum field theory provided a rather unexpected context in which non-commutative Rota--Baxter algebras came to prominence; more precisely, the Hopf algebraic approach to the so-called BPHZ renormalisation program \cite{CK1} (describing a specific subtraction procedure by which the regularised amplitudes corresponding to Feynman graphs can be rendered finite). Since then the range of applications of the theory of non-commutative Rota--Baxter algebras has grown steadily. In \cite{Guo,KuRoYa,Rota3} the reader will find useful introductory surveys. 

\smallskip

We should emphasise that the second part (section 4) of the work at hand does neither strive for an exhaustive presentation of the theory of Rota--Baxter algebras nor does it seek to include a complete list of references. The latter restriction is mainly due to the extensive developments that have taken place during the last ten years, which resulted in numerous ramifications of the theory of Rota--Baxter algebras into other fields. The present survey instead offers some insights into recent progress made in a particular direction -- maybe the most significative one from the point of view of integral and differential calculus -- namely the extension of the fundamental results of the commutative theory to the Hopf algebraic and non-commutative settings. The results presented in this survey were obtained mainly by the authors, various of them together with Jos\'e Gracia-Bond\'\i a and Dominique Manchon, and include, among others, the Bohnenblust--Spitzer identity and the construction of the standard Rota--Baxter algebra (a generalisation to the non-commutative setting of Rota's presentation of the free Rota--Baxter algebra in terms of symmetric functions).

\smallskip

The article is organised into two parts. The first one (Sections \ref{sect:perm} - \ref{sect:chrono}) commences by surveying the Hopf algebraic interpretation of classical iterated integral calculus and its rooting in the geometry of simplices. This part includes the case of tree-shaped iterated integrals, relevant in many situations (for example, time-dependent perturbation theory, non-geometric rough paths, etc.). Chronological calculus and the associated pre-Lie structures are introduced, together with the construction of their enveloping algebras. Fundamental objects are considered, such as the Magnus operator and its inverse, arising from logarithmic and exponential maps.

The second part (Section \ref{sect:RBA}) begins with a very brief historic account on the works by Baxter, Cartier, Chen, Rota, and others. Then the notion of Rota--Baxter algebra is introduced and its physical as well as probabilistic origins are recalled by presenting several classical examples of such algebras that stem mostly (but not exclusively) from analysis. Algebraic structures related to the Rota--Baxter identity, such as (quasi-)shuffle, post- and pre-Lie algebras as well as Atkinson's factorisation are discussed. The latter is then linked in the commutative case to Spitzer's classical identity. The approaches of Cartier and Rota to constructions of free commutative Rota--Baxter algebras are then briefly outlined. Links with the theories of symmetric functions, commutative shuffle algebras and fundamental (Spitzer-type) identities are put forward as well. This presentation of the commutative theory is tuned toward its non-commutative extension as well as the inclusion of Hopf algebraic properties; this approach is neither standard nor, in some sense, the most natural one from the point of view of the commutative setting, but it has the advantage of making transparent the ideas underlying the transition from the commutative to the non-commutative realm.

We enter then the core of the subject, i.e., the theory of non-commutative Rota--Baxter algebras, with a view towards Hopf algebraic properties. Our approach is based mainly on three ideas. First, the existence of non-commutative analogs of most of the fundamental commutative constructions. This idea is illustrated by the close relationship between the structure of free Rota--Baxter algebras and non-commutative symmetric functions \cite{gelfand}. Then, the existence of non-commutative analogs of most of the fundamental commutative formulas is shown. Eventually, we would like to emphasise that calculus in Rota--Baxter algebras is in some sense ``generic'' with respect to several theories, in the sense that calculus in free Rota--Baxter algebras models calculus in various domains, very much like tensor calculus is generic for the theory of iterated integrals and Lie algebras\footnote{In a sense that can be made rigorous, see for instance \cite{Reutenauer}. More precisely, from a technical viewpoint, formulas in tensor algebras can be shown to hold in arbitrary enveloping algebras of graded Lie algebras. Similarly, formulas that hold in free Rota--Baxter algebras hold in any pre-Lie or shuffle algebra.}. Theories to which such a remark applies in the non-commutative framework include the one of iterated integrals and sums of operator valued functions (corresponding respectively to non-commutative shuffle and quasi-shuffle products) as well as of derivations and differential operators (corresponding to pre-Lie products). Concretely, all this points at the observation that it is often convenient to work within the setting of Rota--Baxter algebras when it comes to looking for universal formulas solving problems in integral calculus, difference calculus, dynamical systems, and related fields.

\medskip

All algebraic structures are defined over the ground field $k$ of characteristic zero. They are graded if their structure maps are defined in the tensor category of graded vector spaces, that is, if objects are graded vector spaces and the structure maps respect the grading. Connectedness in the graded sense refers to the degree zero component being isomorphic to the ground field. We also assume any $k$-(co)algebra to be (co)associative, if not stated otherwise. 



\section{Generalised iterated integrals}
\label{sect:perm}

We start by studying algebraic structures underlying calculus with iterated integrals and show how they lead naturally to the notions of descent (Hopf) algebra as well as permutation Hopf algebra. We feature an elementary but powerful observation: since $n$-fold iterated integrals are over a $n$-dimensional simplex, their algebraic properties translate (exactly) into combinatorial and geometric properties of Euclidean simplices. 
With the exception of this geometric point of view originating in \cite{patring} and the resulting presentation, the material is relatively classical and to some extend ``folklore", although seemingly more in algebra and combinatorics than in integral calculus. 

The study of descent algebras was initiated to a large extend by the work of Solomon \cite{solomon}, which gave rise to the theory of non-commutative representations of Coxeter groups. Descents are studied in detail in Reutenauer's classical monograph on free Lie algebras \cite{Reutenauer} and are closely related to the fine structure of graded connected cocommutative Hopf algebras (that is, enveloping algebras of reduced graded Lie algebras) \cite{patgrad}. The Hopf algebra of descents was in fact introduced by Malvenuto and Reutenauer in the seminal paper  on the duality between descents and quasi-symmetric functions \cite{mr}, and independently in \cite{gelfand} in the context of the theory of non-commutative symmetric functions. 

The algebra structure on the linear span of permutations was presented by the second author under the name ``ring of simplices" \cite{patring}. Its Hopf algebra structure was unfolded later in \cite{mr}. It is known in the literature as permutation Hopf algebra, Malvenuto--Reutenauer Hopf algebra or Hopf algebra of free quasi-symmetric functions \cite{duc}. 

The tree-shaped expansion of a quadratic differential equation is taken essentially from reference \cite{bmp}\footnote{For a concrete application of the technique, although to a different problem than the one we use here as a motivation, we refer to the same article.}, where the technique was used in the context of time-dependent perturbation theory. The construction of rooted trees is given mainly in the present article as a simple illustration of a general process of obtaining combinatorial expansions of solutions of (nonlinear) differential equations. This process was used systematically in recent work of Martin Hairer and collaborators to find (renormalized) solutions of stochastic (partial) differential equations \cite{hairer}. 
Although simple, the example of rooted trees is interesting in that it allows to put in evidence various algebraic phenomena. It shows, for example, how integration by parts formula translates in the context of perturbative expansions.


\subsection{Permutations and simplices}
\label{ssect:permutations}

Let us write $S_n$ for the group of permutations of the set $[n]:=\{1,\dots,n\}$ and use the notation $\sigma =(\sigma(1),\ldots, \sigma(n))$ for elements $\sigma \in S_n$. Consider as a starting point the linear $N \times N$ matrix differential equation:
\begin{equation}
\label{matrixIVP}
	\dot{X}(t)=X(t)H(t),\qquad X(0)=Id,
\end{equation}
where $H(t)$ is, say, $C^\infty$ with bounded coefficients. Here, $Id$ denotes the $N \times N$ identity matrix. The solution of such a linear initial value problem is given by a series of iterated integrals, usually referred to in the mathematics literature as Picard series and in the physics literature as Dyson series or time-ordered exponential:
\allowdisplaybreaks
\begin{align}
	X(t) 	&= X_0(t)+X_1(t)+X_2(t)+\cdots+X_n(t) +\cdots \nonumber \\
		&= Id + \int\limits_{\Delta_t^1}dt_1H(t_1)
			 + \int\limits_{\Delta_t^2}dt_1dt_2H(t_1)H(t_2)
			 + \cdots +\int\limits_{\Delta_t^n}dt_1\cdots dt_nH(t_1)\cdots H(t_n)+\cdots.\label{DysonExp}
\end{align}
The integration domains in \eqref{DysonExp} are defined to be the ``standard'' $n$-simplices 
$$
	\Delta_t^n:=\{(t_1,\dots, t_n)\ |\ 0 \leq t_1 \leq \cdots \leq t_n \leq t\}.
$$ 
Given a ($n$-dimensional) convex polytope $P$ in $\RR^n$, we will write, in general, $\langle P|H \rangle$ for the integral of the product $H(t_1) \cdots H(t_n)$ over $P$, $\int \limits_Pdt_1\cdots dt_nH(t_1)\cdots H(t_n)$, so that $X_n(t)= \langle \Delta^n_t |H \rangle$. Note that with this notation 
$$
	\dot{X}_n(t)=X_{n-1}(t)H(t)= \langle \Delta^{n-1}_t | H \rangle \cdot H(t),
$$ 
whereas Fubini's rule reads 
$$
	\langle P\times Q|H \rangle = \langle P|H \rangle \langle Q | H\rangle.
$$ 
In more abstract terms, any time-dependent matrix (or more generally sufficiently regular operator) $H=H(t)$ defines a character, that is, a multiplicative map from the non-commutative algebra of linear combinations of convex Euclidean polytopes equipped with the cartesian product to the algebra of time-dependent matrices. This rule is compatible with geometric decompositions in the sense that if $P=Q\coprod T$, where $P,Q$ and $T$ are $n$-dimensional polytopes, and $Q \cap T$ is contained in a hyperplane, then $\langle P|H \rangle = \langle Q|H \rangle + \langle T|H\rangle$. These rules, together with the integration by parts formula, are enough, as we shall see, to capture the fundamental algebraic and combinatorial properties of classical calculus.

As a motivational example, consider the classical (continuous Baker--Campbell--Hausdorff) problem of expressing the solution $X(t)$ of \eqref{matrixIVP} as a true (matrix) exponential 
$$
	X(t)=\exp(\Omega (t)), \quad \Omega (0)=0.
$$ 
Its answer relies on computing the logarithm of the function $X(t)$, which amounts to calculating the logarithm of the sum of the iterated integrals $X_i(t)= \langle \Delta^i_t |H \rangle$. By Fubini's rule, the formula for  multiplying two such iterated integrals results from the simplicial decomposition of a cartesian product of simplices as a union of simplices
\allowdisplaybreaks
\begin{align*}
	\Delta_t^n \times \Delta_t^m 
	&= \{(t_1,\ldots, t_n,u_1,\ldots ,u_m)\ |\ 0\leq t_1\leq \cdots \leq t_n\leq t,\ 0 \leq u_1\leq \cdots\leq u_m\leq t\}\\
	&=\bigcup\{(y_{\sigma(1)},\ldots ,y_{\sigma(n+m)})\ |\ 0\leq y_1\leq \cdots \leq y_{m+n}\leq t\},
\end{align*}
where the union runs over permutations $\sigma\in S_{m+n}$ such that ${\sigma(1)}<\cdots < {\sigma(n)}$  and ${\sigma(n+1)}< \cdots < {\sigma(m+n)}$. This is the classical rule used, among others, in algebraic topology, to define products in simplicial and singular cohomology. Since ${\sigma(i)}<{\sigma(i+1)}$ with a possible exception for $i=n$, i.e., it may be that $\sigma(n)>\sigma(n+1)$, such a permutation is said to have at most one descent, in position $n$. This yields the following 

\begin{defn}\label{def:descent}
The descent set of a permutation $\sigma \in S_n$ is the subset of $[n-1]$ defined by:
$$
	Desc(\sigma):=\{i\ |\ \sigma (i)>\sigma(i+1),\ 1\le i <n\}.
$$
\end{defn}

\noindent For example, the descent set of the permutation $(3174526) \in S_7$ is 
$$
	Desc(3174526)=\{1,3,5\}.
$$ 
In particular, setting  
$$
	\Delta_{\sigma,t}^n:=\{(t_{\sigma(1)},\ldots, t_{\sigma(n)})\ |\ 0\leq t_1\leq\dots \leq t_{n}\leq t\}
$$
and defining for $\sigma \in S_n$
$$
	X_{\sigma}(t) := \langle \Delta_{\sigma,t}^n |H \rangle 
		      =\int\limits_{\Delta_{\sigma,t}^n}du_1\cdots du_n H(u_1)\cdots H(u_n),
$$
we obtain for the product of iterated integrals
$$
	X_i(t)X_j(t)=\sum\limits_{\sigma\in S_{i+j} \atop Desc(\sigma)\subseteq\{i\}}X_{\sigma}(t).
$$
Compare this with products of integrals following from the classical integration by parts formula
\begin{align*}
	X_1(t)X_1(t) &=\int\limits_{0\leq t_1\leq t}dt_1H(t_1)\int\limits_{0\leq t_{2}\leq t}dt_2H(t_2)\\
			   &= \iint\limits_{0\leq t_1\leq t_{2}\leq t} dt_1dt_2H(t_1)H(t_2)	+ 
			   	\iint\limits_{0\leq t_2\leq t_{1}\leq t} dt_1dt_2H(t_1)H(t_2)\\
			   &= X_{(12)}(t) + X_{(21)}(t),		
\end{align*}
with the descent in position $1$ in the second term.  Similarly,
\begin{align*}
	X_2(t)X_1(t) &=\int\limits_{0\leq t_1\leq t_{2}\leq t}dt_1dt_2H(t_1)H(t_2)\int\limits_{0\leq t_3\leq t}dt_3H(t_3)\\
			   &= \iiint\limits_{0\leq t_1\leq t_{2}\leq t_{3}\leq t} dt_1dt_2dt_3H(t_1)H(t_2)H(t_3)	+ 
			   	\iiint\limits_{0\leq t_1\leq t_{3}\leq t_{2}\leq t} dt_1dt_2dt_3H(t_1)H(t_2)H(t_3)     +\\
			   &\quad	\iiint\limits_{0\leq t_3\leq t_{1}\leq t_{2}\leq t} dt_1dt_2dt_3H(t_1)H(t_2)H(t_3)\\
			   &= X_{(123)}(t) + X_{(132)}(t) + X_{(231)}(t) ,	
\end{align*}
with the descent in position $2$ in the second and third term. 

Iterating products of iterated integrals naturally leads to computing products of non-standard simplices parametrised by permutations. Recall that if $A:=\{a_1,\ldots, a_n\}$ is a set of integers, written in the natural order: $a_1<a_2<\cdots < a_n$, the standardization map $st_{[n]}$ is the unique strictly increasing map from $A$ to $[n]=\{1,\dots,n\}$: $st_{[n]}(a_i)=i$ (one could write $st^A_{[n]}$ for $st_{[n]}$ to emphasize the dependency on $A$, but we stick to a light notation since the dependency on $A$ will always be explicited or be clear from the context further on). We write simply $st$ when no confusion can arise. We then get in general
$$
	\Delta_{\sigma,t}^n\times \Delta_{\beta,t}^m
	=\bigcup\{(y_{\gamma(1)},\dots ,y_{\gamma(n+m)})\ |\ 0\leq y_1\leq\dots \leq y_{m+n}\leq t\},
$$
where the union runs now over permutations $\gamma\in S_{n+m}$ such that $st_{[n]}(\gamma(i))=\sigma(i)$ for $i\leq n$ with $A:=\{\gamma(1),\dots,\gamma(n)\}$ and $st_{[m]}(\gamma(i))=\beta(i-n)$ for $n+1\leq i\leq n+m$ and $A:=\{\gamma(n+1),\dots,\gamma(n+m)\}$. 
Equivalently, writing $\sigma \otimes \beta$ for the permutation in $S_{n+m}$ such that $\sigma \otimes \beta(i) = \sigma(i)$ for $i \leq n$ and $\sigma \otimes \beta(i) = n+\beta(i-n)$ else, we get
$$
	\Delta_{\sigma,t}^n \times \Delta_{\beta,t}^m = \bigcup \Delta_{\nu,t}^{n+m},
$$
where $\nu$ runs over permutations of the form $\alpha\circ (\sigma\otimes\beta)$ with $Desc(\alpha)\subset \{n\}$. 

It is useful to understand these formulas geometrically. Let us denote $pr^+_n$ and $pr_m^-$ the two canonical projections from $\RR^{n+m}$ to $\RR^n$ and $\RR^m$ ($(t_1,\dots,t_{n+m})\longmapsto (t_1,\dots,t_n)$ resp.~$(t_{n+1},\dots,t_{n+m})$). Then, from the standard properties of projections of simplices, we have $pr_n^+(\Delta_{\sigma,t}^n \times \Delta_{\beta,t}^m)=\Delta_{\sigma,t}^n$, respectively~$pr_m^-(\Delta_{\sigma,t}^n \times \Delta_{\beta,t}^m)=\Delta_{\beta,t}^m$. In fact, with the same notation as above, one also has (looking for example at the boundary of $\Delta_t^\nu$)
$$
	pr_n^+(\Delta_{\nu,t}^{n+m})
	=\Delta_{\sigma,t}^n,\quad pr_m^-(\Delta_{\nu,t}^{n+m})
	=\Delta_{\beta,t}^m ;
$$
and this identity is another characterization of permutations $\nu$ of the form $\alpha \circ (\sigma\otimes\beta)$ with $Desc(\alpha)\subset \{n\}$. 

This yields the following definition \cite[Sect. 1.3]{patring}

\begin{defn}\label{def:simplexalg}
The algebra of simplices (or algebra of permutations), denoted $\mathcal S$, is the vector space $\bigoplus_{n\in\NN} \QQ[S_n]$ equipped with the product
\begin{equation}
\label{prodDesc}
	\sigma \ast \beta 
	:= \sum\limits_{\alpha \in S_{n+m} \atop Desc(\alpha)\subseteq \{n\}}\alpha\circ (\sigma\otimes\beta)
	=\sum\limits_{\gamma\in S_{n+m}\atop {st_{[n]}(\gamma(i))=\sigma(i),i\leq n\atop st_{[m]}(\gamma(i))
		=\beta(i-n),i>n}}\gamma 
	=\sum\limits_{\nu \in S_{n+m}\atop {pr_n^+(\Delta_{\nu,t}^{n+m})
		=\Delta_{\sigma,t}^n\atop pr_m^-(\Delta_{\nu,t}^{n+m})
		=\Delta_{\beta,t}^m}}\nu
\end{equation}
for $\sigma\in S_n$ and $\beta\in S_m$.
\end{defn}

Concretely, $\sigma\ast\beta$ is the sum of all permutations $\alpha$ in $S_{n+m}$ such that the elements of the sequence $(\sigma(1),\dots,\sigma(n))$ are ordered as the elements of $(\alpha(1),\dots,\alpha(n))$ (that is, 
$\alpha(i)>\alpha(j)$ if and only if $\sigma(i)>\sigma(j)$) and the elements of the sequence $(\beta(1),\dots,\beta(m))$ are ordered as the elements of $(\alpha(n+1),\dots,\alpha(n+m))$.

Notice that $\mathcal S$ is a graded algebra with unit and that the associativity of the product $\ast$ in \eqref{prodDesc} follows immediately from the associativity of cartesian products in geometry.

\ \par 
The link between the algebraic structure of iterated integrals and the one of simplices and permutations does not only hold when dealing with a single time-dependent operator. Actually, in many applications, say, for example, when expanding the solution of a differential equation such as $\dot{X}(t)=X(t)H_1(t)+H_2(t)X(t)$ (or, more generally, nonlinear differential equations involving several time dependent operators or having more complex structures -- see \cite{AG2} as well as \cite{bmp}, from which the following results are taken, for a concrete example), one has to consider iterated integrals of products of various time-dependent operators. This does not affect the algebraic rules that have just been deviced, as we show now.

Let ${\mathbf L}=(L_1,\ldots,L_n)$ be an arbitrary sequence of time-dependent operators, $L_i=L_i(t)$, $1 \le i \le n$, sufficiently regular (for example, as in \cite{bmp}, assume they are strongly continuous map from $\mathbb{R}$ into bounded self-adjoint operators on a Hilbert space $\mathcal H $). For $\sigma\in S_n$, set:
$$
	{\mathbf L}_\sigma(t):=\int\limits_{\Delta_{\sigma ,t}^n}dt_1\cdots dt_nL_1(t_1)\cdots L_n(t_n).
$$

The notation is extended to linear combinations of permutations, so that for $\mu = \sum_{\sigma\in S_n}\mu_\sigma\sigma$, ${\mathbf L}_\mu(t):=\sum_{\sigma\in S_n} \mu_\sigma{\mathbf L}_\sigma(t).$ For $\mathbf{K}=(K_1,\ldots,K_m)$ another sequence of time-dependent operators, we write $\mathbf{L}\cdot \mathbf{K}$ for the concatenation product $(L_1,\ldots,L_n,K_1,\ldots,K_m)$. Then, for $\alpha \in S_n$ and $\beta \in S_m$ two permutations it follows from the combinatorial formula for the cartesian product of simplices that

\begin{lem}
\label{productfla}
We have:
$$
	{\mathbf L}_\alpha(t){\mathbf L}_\beta(t)
	= {\mathbf L}_{\alpha\ast\beta}(t).
$$
\end{lem}

We also mention the following lemma, which is useful in nonlinear contexts and which can be generalised to other forms of products of operator sequences:
	
\begin{lem}\label{fundam}
We have, for $\mathbf L$ and $\mathbf K$ as above and $J=J(t)$ a time-dependent operator:
$$
	\int_{0}^{t} \dd s{\mathbf L}_\alpha(s)J(s)\mathbf K_\beta(s)
	=\sum_\gamma({{\mathbf L}\cdot(J)\cdot {\mathbf K}})_\gamma(t),
$$
where $\gamma$ runs over the permutations in $S_{n+m+1}$ with $\gamma(n+1)=n+1$, $st(\gamma(1),\ldots,\gamma(n))=\alpha$, $st(\gamma(n+2),\ldots,\gamma(n+m+1))=\beta$.
\end{lem}

The proof follows from the same arguments that lead to the combinatorial expression for a cartesian product of simplices, we refer to \cite{bmp} for details.


\subsection{Descents, NCSF and the BCH formula}
\label{ssect:descents}

In many applications in integral calculus, only iterated products of standard simplices $\Delta_t^n$ appear. It is therefore not necessary to appeal to the full algebra $\mathcal S$, but only to the smaller algebra generated by the identity elements in the various $S_n$. This algebra is known as the \it descent algebra \rm and plays a key role in the theory of free Lie algebras \cite{Reutenauer} as well as in the structure theory of Hopf algebras \cite{patgrad}.

\begin{defn}
The descent algebra $\mathcal{D}$ is the graded unital subalgebra of $\mathcal S$ generated by the identity permutations $1_n:=(1, \ldots, n)\in S_n$, $n\geq 1$.
\end{defn}

Notice that, geometrically, elements of $\mathcal{D}$ are therefore linear combinations of cartesian products of standard simplices; the associated integrals are products of iterated integrals over standard simplices.

Generalizing what happens when considering the cartesian product of two standard simplices, to each subset $A$ of the set $[n-1]$ are associated two elements in $\QQ [S_n]$. They are referred to as ``Solomon elements'' since they appear as a key ingredient in Solomon's theory of non-commutative representations of the symmetric groups \cite{solomon}:

\begin{defn} 
For $A \subset [n-1]$, the Solomon elements $D_{=A}$ and $D_{A}$ in $\QQ[S_n]$ are defined by:
$$
	D_{=A}:=\sum\limits_{\sigma\in S_n \atop Desc(\sigma)=A}\sigma,
	\qquad 
	D_{A}:=\sum\limits_{\sigma\in S_n \atop Desc(\sigma)\subseteq A}\sigma.
$$
To avoid ambiguities, elements will be written $D_{=A}^{(n)}$, respectively $D_{A}^{(n)}$, when their belonging to $\QQ[S_n]$ has to be emphasised.
\end{defn}

From the definition of the product \eqref{prodDesc} in $\mathcal S=\bigoplus_{n\in\NN} \QQ[S_n]$, the permutations that appear in the expansion of $D_{i_1,i_1+i_2,\ldots,i_1+\dots+i_{n-1}}^{(i_1+\cdots+i_n)}$ parametrise the decomposition into a union of simplices of the product of the simplices $\Delta_t^{i_1},\ldots,\Delta_t^{i_n}$. In particular their sum belong to $\mathcal D_n$, and we have

\begin{lem}\label{pdtfla}
In the descent algebra we have
$$
	1_{i_1}\ast \cdots \ast 1_{i_n} = D_{i_1,i_1+i_2,\ldots,i_1 + \cdots + i_{n-1}}^{(i_1 + \cdots + i_n)}
$$
and
$$
	D_S^{(n)}\ast D_T^{(m)}=D_{S\cup\{n\}\cup (T+n)}^{(n+m)}.
$$
\end{lem}

Since the (non-empty) sets $\{\sigma\in S_n \ |\  Desc(\sigma)=A\}$, where $A$ runs over subsets of $[n-1]$, define a partition of $S_n$, the elements $D_{=A}$ are nonzero and linearly independent in $\QQ [S_n]$. By  M\"obius inversion (in the poset of subsets of $[n-1]$), the $D_A$ are also linearly independent. Moreover, since the M\"obius inversion coefficients of the poset of subsets of $[n-1]$ are $(-1)^{|A|-|B|}$ for $B\subset A$, we have

\begin{prop}\label{moebius}
The two families $D_{=A}$ and $D_{A}$ are two linear basis of ${\mathcal D}_n$ related by the formulas:
$$
	D_A=\sum\limits_{B\subseteq A}D_{=B},
	\qquad 
	D_{=A}=\sum\limits_{B\subseteq A}(-1)^{|A|-|B|}D_{B}.
$$
\end{prop}

From Proposition \ref{moebius} and Lemma \ref{pdtfla} we get

\begin{cor}\label{cor:free}
The descent algebra is a free associative algebra with unit generated by the $1_n$, $n\geq 1$.
\end{cor}

\begin{rmk}
In the context of the category of graded connected cocommutative Hopf algebras, a notion of descent algebra is defined as an algebra of natural transformations of the forgetful functor to graded vector spaces. It is generated by the projections from such a Hopf algebra to its graded components. The product is induced by the convolution product of linear endomorphisms. The notion of descent algebra defined in that way is isomorphic to $\mathcal D$ (but the interpretation of its elements as linear combinations of permutations is not availble since there is in general no canonical action of symmetric groups on a graded connected cocommutative Hopf algebra) \cite{patgrad}. As a free associative algebra over a set of graded generators, the descent algebra is canonically isomorphic to the algebra of non-commutative symmetric functions {\rm{(NCSF)}} \cite{gelfand}. These two points of view have proven to be very useful to investigate the properties and applications of $\mathcal D$.
\end{rmk}

A classical application of the rules governing the descent algebra is the explicit description of the so-called continuous Baker--Campbell--Hausdorff (BCH) formula:

\begin{thm}\label{thm:bch}
For the series \eqref{DysonExp} we have, using the notation $X_{\sigma}(t)=:X_n(t)\cdot \sigma$:
\begin{align}
	\Omega (t) := \log(X(t)) 
		&=\sum\limits_{n=1}^\infty X_n(t)\cdot \Big(\sum\limits_{S\subset [n-1]} 
			\frac{(-1)^{|S|}}{|S|+1} D_S^{(n)} \Big) \nonumber \\ 
		&=\sum\limits_{n=1}^\infty X_n(t)\cdot \Big(\sum\limits_{S\subset [n-1]} 
			\frac{(-1)^{|S|}}{n}{{n-1}\choose{|S|}}^{-1} D_{=S}^{(n)}\Big). \label{BCH}
\end{align}
\end{thm}

\begin{proof}
Recall that $X_n(t) \cdot \sigma := X_\sigma(t)$. The first identities follow immediately from the product formulas for standard simplices and iterated integrals. Let us expand the $D_S$ in terms of the $D_{=S}$. The coefficient of $D_{=T},T\subset [n-1],|T|=l$, in this expansion is given by:
\begin{align*}
	\sum\limits_{T \subset S\subset [n-1]}\frac{(-1)^{|S|}}{|S|+1}
	&=\sum\limits_{j=0}^{n-l-1}{n-l-1\choose j}\frac{(-1)^{l+j}}{l+j+1}\\
	&=(-1)^{l}\sum\limits_{j=0}^{n-l-1}\int_0^1(-1)^j{n-l-1\choose j}x^{l+j}dx\\
	&=(-1)^{l-1}\int_0^1(1-x)^{n-l-1}  x^{l}dx=(-1)^{l}\frac{1}{n}{n-1\choose l}^{-1}.
\end{align*}
From this the formula follows.
\end{proof}

\begin{rmk}\label{rmk:Casas}
The series $\Omega(t)$ is a Lie series \cite[Chap. 3]{Reutenauer}. Concretely, this property means that $\Omega(t)$ can be expressed as a weighted sum of iterated integrals of nested commutators
$$ 
	[X]_\sigma(t):=X_n(t)\cdot [\sigma(1),[\sigma(2),[\cdots [\sigma(n-1), \sigma(n)],\cdots ]]].
$$ 
Here, we identify a permutation $\beta\in S_n$ with the word $\beta(1)\cdots \beta(n)$ and the bracket of two words is defined  by $[w,w']:=ww'-w'w$, so that, for example, for $\sigma=(213)$, $[2,[1,3]]=[2,13-31]=213-231-132+321$.
Note that in the latter expansion there is a unique word ending with $3=\sigma(3)$. This property holds in general and is straightforward to see. Indeed, when $[\sigma(1),[\sigma(2),[\cdots [\sigma(n-1), \sigma(n)],\cdots ]]]$ is expanded as a sum of words, the word $\sigma(1)\cdots\sigma(n)$ is the only one ending with $\sigma(n)$ that appears. 
This observation has two classical consequences \cite[Sect. 5.6.2]{Reutenauer}: 
\begin{itemize}

\item the brackets $[\sigma(1),[\sigma(2),[\cdots [\sigma(n-1),n],\cdots ]]]$, where $\sigma\in S_{n-1}$, are linearly independent in $\QQ[S_n]$ 

\item since the space $P_n$ of multilinear Lie polynomials (that is, the vector space generated by the iterated brackets $[\sigma(1),[\sigma(2),[\cdots [\sigma(n-1), \sigma(n)],\cdots ]]]$) is of dimension $(n-1)!$, the brackets $[\sigma(1),[\sigma(2),[\cdots [\sigma(n-1),n],\cdots ]]]$, $\sigma\in S_{n-1}$ form a basis of $P_n$. For any element $x$ of $P_n$, $x= \sum\limits_{\sigma\in S_n} x_\sigma \cdot \sigma$, the expansion of $x$ in this basis is then 
$$
	x=\sum\limits_{\sigma\in S_{n-1}}x_{\sigma^+}[\sigma(1),[\sigma(2),[\cdots [\sigma(n-1),n],\cdots ]]],
$$
where $\sigma^+=(\sigma(1),\dots,\sigma(n-1),n)$.
\end{itemize}
Similar results hold by using left-to-right iterated brackets or fixing as last letter another letter (for example $1$ instead of $n$). This result has an interesting consequence as it implies \cite{Arnal} 
\end{rmk}

\begin{cor}
The series $\Omega(t)$ rewrites 
$$
	\Omega (t)  
	=\sum\limits_{n=1}^\infty \sum\limits_{\sigma\in S_{n-1}} 
	\frac{(-1)^{|Desc(\sigma)|}}{n}{{n-1}\choose{|Desc(\sigma)|}}^{-1}[X]_{\sigma^+}.
$$
\end{cor}


\subsection{Rooted trees and nonlinear differential equations}
\label{sect-tree}

We have seen in the previous sections how to handle algebraically products of iterated integrals parametrised by permutations. In various situations, such as when dealing with nonlinear differential equations, it is useful to consider more general iterated integrals parametrised by rooted trees. Here, we'll treat the example of planar binary trees, following reference \cite{bmp}, whose developments were motivated by time-dependent perturbation theory. As a motivation, the reader may consider the perturbative expansion of the solution of the quadratic equation $\dot{X}(t)=X(t)H(t)X(t)$. The equations governing the Hamiltonians and effective Hamiltonians of time-dependent perturbation theory are more complex, but obey the same combinatorics. The following constructions can be generalised to other nonlinear equations, our presentation should be enough to aim at the general principle.

\smallskip 

Consider again the time-dependent operator $H(t)$ and, for example, the iterated integral
$$
	\int\limits_{0\leq t_1,t_3\leq t_2 \leq t}dt_1dt_2dt_3H(t_1)H(t_2)H(t_3).
$$
The integration domain $\{(t_1,t_2,t_3)\ |\ 0\leq t_1,t_3\leq t_2\leq t\}$ can be parametrised by the poset structure on the set $\{1,2,3\}$, defined by $2 << 1, \ 2<< 3$ (in general, the rule to define the poset structure from the iterated integral is $t_i\leq t_j\implies j<<i$). The associated Hasse diagram is a rooted tree (a connected and simply connected graph with a distinguished vertex called the root, the graph is directed according to the ordering) with $2$ as a root and two branches relating $2$ and $1$ and $2$ and $3$. The domain splits, moreover, as the union of two simplices: $\{(t_1,t_2,t_3)\ |\ 0\leq t_1\leq t_3\leq t_2\leq t\}=\{(t_1,t_3,t_2)\ | \ 0\leq t_1\leq t_2\leq t_3\leq t\}$ and $\{(t_1,t_2,t_3)\ | \ 0\leq t_3\leq t_1\leq t_2\leq t\}=\{(t_2,t_3,t_1)\ |\ 0\leq t_1\leq t_2\leq t_3\leq t\}$.

Let us generalise this construction. Consider a planar binary tree $T$ with $n$ vertices. Recall that planar means that the tree is drawn in the plane with the root at the bottom and that at each vertex the outgoing branches are ordered -- say, from left to right. Planar binary means that each vertex has two outgoing branches. At internal vertices $v$, one of these branches may be empty, by what we mean that it does not link $v$ to another vertex. At leaf vertices (top vertices), the two branches are empty.
A planar binary tree can be represented uniquely as the grafting of two planar binary trees, $T=T_1 \vee T_2$, where $T_1$ and $T_2$ are the branches attached to the root. Recall that empty branches are allowed, but we distinguish the case where the left or the right branche is empty (so that an outgoing branch is always pointing to the left or the right). Notationally, we distinguish
between $T_1 \vee \emptyset$ and $\emptyset \vee T_1$. The tree is then labelled from left to right and top to bottom by elements from the set $\{1, \ldots ,n\}$. When performed inductively, the labelling of a tree $T_1 \vee T_2$ with $n=k+1+l$ vertices ($T_1$ and $T_2$ having $k$ respectively $l$ vertices) by a subset $S$ of the integers of cardinality $n$ is performed as follows. The tree $T_1$ is labelled by the first $k$ elements in $S$, the root is labelled by the $k+1$-th element of $S$, the tree $T_2$ is labeled by the remaining $l$ elements.

Using such a labelling of the vertices, we can view the tree $T$ as the Hasse diagram of a (well-defined) order on $[n]$, denoted $<_T$. We write $\Delta_T$ for the associated subset of the cube $[0,t]^n\subset \RR^n$ defined by the inequalities $t_i\leq t_j\Longleftrightarrow j\leq_Ti$. We then set
$$
	X_T(t)= \langle \Delta_T|H \rangle 
	=\int_{\Delta_T}dt_1\cdots dt_nH(t_1)\cdots H(t_n).
$$
The notation $X_T(t)$ is extended to linear combinations of trees, so that, for example, if $Z=T+2T'$, where $T$ and $T'$ are two arbitrary trees, then $X_Z(t):=X_T(t)+2X_{T'}(t)$.

As the root of the tree is associated to the highest coordinate of the points in $\Delta_T$ and since the labels of $T_1$ and $T_2$ are incomparable for $<_T$, we get

\begin{lem}
For $T=T_1\vee T_2$ we have that
$$
	X_T(t)=\int_0^tds X_{T_1}(s)H(s)X_{T_2}(s).
$$ 
or
$$
	\dot{X}_T(t)=X_{T_1}(t)H(t)X_{T_2}(t).
$$
\end{lem}

\begin{prop}\label{produ}
Let $T=T_1\vee T_2$, $U=U_1\vee U_2$ be two planar binary trees, we have:
\begin{equation*}
	X_T(t)X_U(t)=\int\limits_0^tds X_T(s)X_{U_1}(s)H(s)X_{U_2}(s)
				+\int\limits_0^tds X_{T_1}(s)H(s)X_{T_2}(s)X_U(s).
\end{equation*}
In the formula, one or several of the trees $T_1,T_2,U_1,U_2$ may be the empty tree.
\end{prop} 

\begin{proof}
The proof follows from the integration by parts formula.
\end{proof}

Proposition~\ref{produ} can be rephrased geometrically: define recursively a product $\ast$ on planar binary trees by
\begin{equation}
\label{shuffleProd}
	T\ast U=(T\ast U_1)\vee U_2 + T_1\vee (T_2\ast U),
\end{equation}
then 
$$
	X_T(t)X_U(t)=X_{T\ast U}(t).
$$
Using a self-explaining notation for $\Delta_{T\ast U}$, which is actually a union of polytopes associated to trees, we have
$$
	\Delta_T(t)\times\Delta_U(t)=\Delta_{T\ast U}(t).
$$

\begin{cor}
The product \eqref{shuffleProd} gives a combinatorial description of the cartesian product of the domains $\Delta_T$ encoded by trees. It provides the linear span of the set of planar binary trees, $PBT$, with the structure of an associative algebra.
\end{cor}

This follows by direct inspection of the formulas; the associativity follows from the associativity of the cartesian product of polytopes.

\begin{rmk}
In other terms, the integration by parts formula induces an algebra structure on the set of planar binary trees (about the latter, see Remark~\ref{embPBT2} below). Integration by parts can be axiomatised through the notion of (weight $0$) Rota--Baxter algebra, that will be studied in the second part of this article. We will encounter then various phenomena reminiscent of what happens here with planar binary trees.
\end{rmk}

Iterated integrals parametrised by trees and iterated integrals parametrised by permutations are related through the canonical simplicial decompositions of the domains $\Delta_T$.
Let us write $S_T$ for the set of total orders refining $\leq_T$ (that is, if $<<$ is a total order in $S_T$, $i\leq_Tj\Rightarrow i<<j$) and identify a total order on $[n]$ with a permutation according to the rule: $2>>4>>1>>3$ is identified with $\sigma\in S_4$ with $\sigma^{-1}(1)=2,\sigma^{-1}(2)=4,\sigma^{-1}(3)=1,\sigma^{-1}(4)=3$. In general, the permutation $\beta\in S_n$ corresponds to the total order on $[n]$: $\beta^{-1}(1)>>\beta^{-1}(2)>>\dots >>\beta^{-1}(n).$

\begin{lem} 
The domain $\Delta_T$ is the union of the simplices $\Delta_{\sigma}$, where $\sigma$ runs over $S_T$, and we get
$$
	X_T(t):=\int\limits_{\Delta_T}dt_1\cdots dt_n H(t_1)\cdots H(t_n)
		   =\sum\limits_{\sigma \in S_T} X_{\sigma}(t).
$$
\end{lem}

We invite the reader to check the statement in the lemma. Here is an example for the order $2<_T1$, $2<_T3$. We find that 
\begin{align*}
	\Delta_T&=\{(t_1,t_2,t_3)\ |\ 0\leq t_1,t_3\leq t_2\leq t\}\\
	&=\{(t_1,t_2,t_3)\ |\ 0\leq t_1\leq t_3\leq t_2\leq t\}\cup \{(t_1,t_2,t_3)\ |\ 0\leq t_3\leq t_1\leq t_2\leq t\}\\
	&=\{(t_1,t_3,t_2)\ |\ 0\leq t_1\leq t_2\leq t_3\leq t\}\cup \{(t_2,t_3,t_1)\ |\ 0\leq t_1\leq t_2\leq t_3\leq t\}
\end{align*} 
and $S_T=\{1>>3>>2,3>>1>>2\}=\{(1,3,2),(2,3,1)\}$. That is,
$$
	X_T(t)
	=\int\limits_{\Delta_T}dt_1\cdots dt_n H(t_1)H(t_2)H(t_3)
	=X_{(1,3,2)}(t) + X_{(2,3,1)}(t).
$$

 We also get directly as a by-product of the geometrical analysis the following
 
\begin{cor}\label{embPBT}
The map $\kappa$ from $PBT$ to $\mathcal S$ 
$$
	T\longmapsto \sum\limits_{\sigma\in S_T}\sigma
$$
induced by the simplicial decomposition
$$
	\Delta_T(t)=\bigcup\limits_{\sigma\in S_T}\Delta_\sigma(t)
$$
is an (injective) algebra map.
\end{cor}

\begin{rmk}\label{embPBT2}
The map from trees to permutations in the last corollary generalises to arbitrary trees, not necessarily binary ones, and is sometimes referred to as ``arborification process'', a terminology introduced by Jean Ecalle in the theory of dynamical systems \cite{fm} (although the idea of associating to a tree, viewed as a preorder, the set of all its strict refinements --linearisations-- is likely to be much older). The arborification process allows to view a planar binary tree as the sum of its linearisations and to identify $PBT$ with a subspace of $\mathcal S$. This point of view, that was followed in \cite{lr}, motivates further the proposition saying that  the space $PBT$ is stable by the product in $\mathcal S$, and is therefore a subalgebra of the algebra of permutations. The algebra $PBT$ can also be obtained and studied from the point of view of binary search trees
\cite{hnt}.
\end{rmk} 

\begin{rmk} 
In analysis, Bloch sequences, Dyck paths, bracketings and non-crossing partitions, which are in bijection with planar binary trees, have also been used to describe the perturbative expansions in time-dependent perturbation theory (see \cite{d+}, where the equivalence between the various points of view is detailed). The advantage of using trees in iterated integral calculus together with perturbative expansions with respect to other families of combinatorial objects is that it is relatively straightforward to extend their use to more complex situations, as illustrated for example by their recent use in the study of stochastic (partial) differential equations, where perturbative solutions are indexed by complex decorated trees \cite{hairer}.
\end{rmk}

 
\subsection{Flows and Hopf algebraic structures}
\label{ssect:Hopf}

Recall that a bialgebra $H$ is both an associative algebra as well as a coassociative coalgebra, such that the coproduct of the latter is a morphism of algebras from $H$ to $H \otimes H$ or, equivalently, such that the product of the former is a morphism of coalgebras from $H \otimes H$ to $H$. A Hopf algebra is a bialgebra equipped with an antipode, the convolution inverse to the identity map. Commutative Hopf algebras appear naturally as algebras of functions on groups and, conversely, an affine algebraic group is the spectrum of a commutative Hopf algebra. We omit details and refer the reader to the various textbooks and surveys on the subject \cite{Cartier3,Radford}. In the study of differential equations, Hopf algebra structures appear naturally in relation to the group structures underlying flow maps. This observation is one of the many reasons motivating the definition of a Hopf algebra structure on the descent algebra or on the algebra of simplices/permutations.

Notice first that the initial value problem, $\dot{X}(t)=X(t)H(t)$, $X(0)=Id$, can be solved on the interval $[0,t]$ as a succession of solutions: $X(t_1),\ldots ,X(t_n)=X(t)$ with $0<t_1<\cdots<t_n=t$. The $i$-th solution is then obtained by solving the differential equation on the subinterval $[t_{i-1},t_i]$ with initial condition $X(t_{i-1})$. For example, given $0 < s < t$ we have that 
\begin{eqnarray*}
	X(t)   &=&X(s)\Big(Id + \int\limits_{\Delta_{t,s}^1}dt_1H(t_1)+ 
		\cdots +\int\limits_{\Delta_{t,s}^n}dt_1\cdots dt_nH(t_1)\cdots H(t_n)+\cdots \Big)\\
		&=& \Big(Id + \int\limits_{\Delta_{s,0}^1}dt_1H(t_1)+ 
		\cdots +\int\limits_{\Delta_{s,0}^n}dt_1\cdots dt_nH(t_1)\cdots H(t_n)+\cdots \Big)\\ 
		&& \quad \Big(Id + \int\limits_{\Delta_{t,s}^1}dt_1H(t_1)+ 
		\cdots +\int\limits_{\Delta_{t,s}^n}dt_1\cdots dt_nH(t_1)\cdots H(t_n)+\cdots \Big),
\end{eqnarray*}
where $\Delta_{t,s}^n:=\{(t_1,\dots,t_n)|s\leq t_1\leq \dots\leq t_n\leq t\}$. This observation, known as Chen's rule, is helpful in practice, especially in numerical analysis as well as in the theory of rough paths, where it allows to ``control'' the construction of rough paths \cite{max}.

This idea leads to a rule of thumb for defining coproducts for Hopf algebra structures according to ``time-orderings''. Starting from a domain parametrised by $0 \leq t_1 \leq \cdots \leq t_n\leq t$ and given $0\leq i\leq n$ together with an arbitrary $s \leq t$, we define the two subsets $0 \leq t_1 \leq \cdots \leq t_i \leq s$ and $s\leq t_{i+1}\leq \cdots \leq t_n\leq t$. When applied, for example, to a permutation $\sigma\in S_n$, starting from 
$$
	\Delta_{\sigma,t}^n=\{(t_{\sigma(1)},\dots,t_{\sigma(n)}|0\leq t_1\leq \dots\leq t_n\leq t\},
$$
Chen's rule leads to the definition of a coproduct on $\mathcal S$:
\begin{equation}
\label{MRcoproduct}
	\Delta(\sigma):=\sum\limits_{i=0}^n\sigma_{|i}\otimes st(\sigma_{|>i}).
\end{equation}
Here, we consider a permutation written as the sequence of its values, and $\sigma_{|i}$ stands for the subsequence associated to the first $i$ integers, and $\sigma_{|>i}$ stands for the subsequence associated to the integers strictly greater than $i$. The map $st$ shifts the sequence $\sigma_{|>i}$ of integers by adding $-i$. For example, for the permutation $\sigma=(3,8,1,4,2,5,7,6)$ and $i=4$ we get $\sigma_{|i}=(3,1,4,2)$, $\sigma_{|>i}=(8,5,7,6)$ and $st(\sigma_{|>i})=(8-4,5-4,7-4,6-4)=(4,1,3,2)$. 

As the operation of dividing a sequence into blocks is coassociative, these rules define a coassociative coproduct on $\mathcal S$ (see \cite{mr} for a combinatorial proof):

\begin{thm}
The coproduct \eqref{MRcoproduct} equips the algebra $\mathcal S$ with the structure of a Hopf algebra.
\end{thm}

The Hopf algebra $\mathcal S$ is neither commutative nor cocommutative. It is known under several names, i.e., the Hopf algebra of permutations, the Malvenuto--Reutenauer Hopf algebra, or the Hopf algebra of free quasi-symmetric functions. 

Since $\mathcal D$ is a subalgebra of $\mathcal S$ freely generated by the identity permutations $1_n$ and since 
$$
	\Delta(1_n)=\sum\limits_{i=0}^n1_i \otimes 1_{n-i},
$$ 
it forms a Hopf subalgebra of $\mathcal S$.

\begin{cor}
The coproduct \eqref{MRcoproduct} restricts to $\mathcal D$ and equips the descent algebra with the structure of a Hopf algebra.
\end{cor}

Note that, as a Hopf algebra, the descent algebra is cocommutative since it is generated as an algebra by the $1_n$ on which the coproduct is cocommutative. 

\ \par

Similar arguments, using Chen's rule to construct coproducts, apply to planar binary trees. In combinatorial terms, the coproduct obtained that way is simply the restriction to $PBT$ of the one on $\mathcal S$, when trees are identified with the sum of their linearisations \cite{hnt,lr}. The simplest way to understand the construction of this coproduct on trees is however via the graphical representation. Since the bottom-to-top ordering of the vertices reflects decreasing times, the coproduct is obtained by splitting in all possible ways a tree $T$ as the union of a subtree containing the root and the family of the remaining trees. 

More formally, let us say that a family $V=\{v_1,\ldots ,v_k\}$ of vertices of $T$ distinct from the root is admissible if they cannot be compared with respect to the order $<_T$, that is, $v_i <_T v_j$ cannot hold for any pair of vertices in $V$. We assume that the family $V$ is ordered according to the labels put on the vertices (from left to right and top to bottom in $T$; notice that this is a different order from the one obtained by restriction of $<_T$) and write (slightly abusively) $T_i$ for the subtree of $T$ having $v_i$ as its root. As a poset, $T_i$ is the set of all vertices $v$ such that $v_i <_T v$. We write $T_V$ for the subtree containing the root, remaining after the $T_i$ have been pruned, and we denote by $T'_V$ the product $T_1 \ast \cdots \ast T_k$. In geometrical terms, this last product corresponds to a cartesian product.

\begin{cor}\label{thm:treecoprod}
With the above notations in place, the coproduct \eqref{MRcoproduct} restricts to $PBT$ and equips the algebra of planar binary trees with the structure of a Hopf algebra. When acting on trees, the coproduct reads
$$
	\Delta(T):=1\otimes T+\sum\limits_V T_V\otimes T'_V,
$$
where the sum runs over all admissible subsets of vertices of the tree $T$.
\end{cor}

\newpage


\section{Advances in chronological calculus}
\label{sect:chrono}

Chronological algebras and time-ordered products appear in (almost) uncountably many places, especially in theoretical physics and control theory. They can be understood as an abstract way of encoding integration by parts. Chronological calculus augments the tools of classical algebra for analysis (such as associative as well as Lie algebra techniques), by introducing the notions of shuffle and pre-Lie products. The former refines the notion of associative product, while the latter augments that of Lie bracket. A powerful framework for chronological calculus is the general theory of -- weight zero -- Rota--Baxter algebras, that will be developed in the following section of the article. 

In a nutshell, Rota--Baxter algebras of weight zero axiomatise the notion of associative algebra equipped with an integral-type operator $R$ satisfying integration by part rule: 
$$
	R(x)R(y)=R(xR(y))+R(R(x)y)).
$$
The general notion of Rota--Baxter algebra allows for the existence of a remainder term 
$$
	R(x)R(y)=R(xR(y))+R(R(x)y) - \theta R(xy)
$$ 
weighted by the scalar $\theta$. Examples of maps satisfying the Rota--Baxter relation of weight $\theta$ are provided by re-summing series or when dealing with generalised integral operators. 

Before turning to the general theory of Rota--Baxter algebras, we will show how the underlying techniques apply in the usual formalism of chronological calculus. Attributing properly the origin of this formalism, which evolved along 70 years of developments in theoretical physics, control theory and combinatorics would be rather difficult and lies outside the scope of the present survey. Among many others, the names of Chen, Feynman, Fliess, Kawski, Rota and Sch\"utzenberger should certainly be quoted (the latter, since it seems that his introduction of half-shuffles in \cite{schutz} in the context of free Lie algebras was actually an elaboration of the notion of time-ordered products as studied by Russian physicists -- hence the notation $T$ for these products in his article\footnote{Hoang Ngoc Minh, private communication.}). 

The development presented in this part of the article is mostly obtained by revisiting the seminal work of Agrachev and Gamkrelidze \cite{AG} who introduced and featured the role of pre-Lie algebra structures in integral/differential calculus. This should be seen in the light of later works. Chapoton and Livernet \cite{ChaLiv} studied the operad of free pre-Lie algebras, Oudom and Guin constructed their enveloping algebras \cite{go}. The first author together with Dominique Manchon \cite{EFM2} used shuffle and pre-Lie algebras in the theory of differential equations and the Magnus formula. In \cite{EFP2014} we studied the interpretation of the structure of the enveloping algebra of a pre-Lie algebra by means of Feynman's time-ordered products. The second author together with Fr\'ed\'eric Chapoton \cite{ChapPat} studied the exponential construction and interpretation of the Magnus operator and its inverse.

In what follows we will extend the results of \cite{patring} which relate the combinatorics of simplices with the theory of free Lie algebras, and show how another fundamental construction, the Eilenberg--MacLane and Sch\"utzenberger ``half-shuffle'' products \cite{schutz}, provides a way to relate chronological calculus with simplicial geometry.


\subsection{Chronological calculus and half-shuffles}
\label{ssect:chonocalc}

Chronological calculus is based on the idea of time-ordering of operators. Consider for example two time-dependent operators, $M(t)$ and $N(t)$, with $M(0)=N(0)=0$, in a non-unital algebra $A$ of operators -- having suitable regularity properties allowing to compute derivatives, integrals, and so on. We say that the product $\dot{M}(s)\dot{N}(u)$ is causal if $s > u$ and anti-causal if $s < u$. Note that we will consider integrals of these products, therefore the equal-time, or diagonal term, $\dot{M}(s)\dot{N}(s)$ is irrelevant. The classical integration by parts rule 
$$
	M(t)N(t)
	= \int_0^t ds \int_0^s du \dot{M}(s)\dot{N}(u) + \int_0^tds \int_0^s du \dot{M}(u)\dot{N}(s)
	=:(M \nwarrow N)(t)+(M \nearrow N)(t)
$$
can then be viewed as a way to decompose the usual operator product into the sum of a ``causal'' product and an ``anti-causal'' product. In this sense, the $\nwarrow $ and $\nearrow$ symbols indicate in what direction time increases, i.e., from right to left in $(M \nwarrow N)(t):=\int_0^t dM(s)N(s)$ and from left to right in $(M \nearrow N)(t):=\int_0^t M(s)dN(s)$. Let us mention that the idea of time-ordering seems to originate in theoretical physics, for example, when defining Feynman propagators respectively Feynman's time-ordered product (further below we will return to the latter).

It is useful to understand this construction of operator products in more abstract terms. Recall that a left (resp.~right) $B$-module, or representation of an associative algebra $B$, is a vector space $V$ equipped with a left (resp.~right) action of $B$ (written as $\cdot$-operation) such that $a \cdot (b \cdot v)=(ab)\cdot v$ (resp.~$(v \cdot b)\cdot a = v \cdot (ba)$) for arbitrary elements $a,b\in B$ and $v \in V$. Returning to time-dependent operators, the associativity of the operator product in $A$, i.e., $M(NP)=(MN)P$, can then be formulated by considering $A$ as a bimodule, i.e., as a left and right module, over itself for the operator product. The integration by parts rule leads to the decomposition of the product into two commuting right and left actions of $A$ on itself, denoted as above in terms of the symbols  $\nwarrow$ and $\nearrow$. Beside indicating the direction of increasing time, the symbols also point in the direction of the action, i.e., in $M \nwarrow N$, the operator $N$ acts from the right on $M$. Indeed, we verify quickly that
\begin{align*}
	((M \nwarrow N) \nwarrow P)(t)	&=\int_0^tds\dot{M}(s)N(s)P(s)=(M \nwarrow NP)(t)\\
	(P \nearrow (N \nearrow M))(t)	&=\int_0^tdsP(s)N(s)\dot{M}(s)=(PN \nearrow M)(t)\\
	((P \nearrow M) \nwarrow N)(t) 	&=\int_0^tdsP(s)\dot{M}(s)N(s)=(P \nearrow (M \nwarrow N))(t).
\end{align*}

In particular, using this notation, the associativity, $(MN)P=M(NP)$, can be formulated as
$$
	(M \nwarrow N) \nwarrow P+(M \nearrow N) \nwarrow P+(MN) \nearrow P
	=M \nwarrow (NP)+M \nearrow (N \nwarrow P)+M \nearrow (N \nearrow P),
$$
and splits into the three operator identities 
\begin{equation}
\label{EMids}
	(M \nwarrow N) \nwarrow P = M \nwarrow (NP),\ 
 	(M \nearrow N) \nwarrow P = M \nearrow (N \nwarrow P),\ 
	(MN) \nearrow P   = M \nearrow (N \nearrow P).
\end{equation}
It turns out that the above notions of causal and anti-causal products, $\nwarrow$ and $\nearrow$, together with these three relations happen to coincide with the axiomatisation of shuffle products in algebraic topology. The latter was one (among many) of the seminal ideas of Eilenberg and MacLane \cite[Sect. 5]{EM}, who used the three relations above to give an abstract proof of the associativity of the shuffle product in a simplicial framework. 

This connection between chronological calculus and algebraic topology is in our opinion illuminating. Indeed, it explains how the causal and anti-causal products are computed when dealing with iterated integrals and reads as follows. We adapt here the definition of half-shuffle products of general simplicial objects to the context of the present paper, that is we define them in the (geometrical realisation of the) algebra of simplices. Recall that given two permutations $\sigma\in S_n$, $\beta\in S_m$, the shuffle product of the corresponding simplices induced by the decomposition of their cartesian product reads
$$
	\Delta_{\sigma,t}^n\times\Delta_{\beta,t}^m
	=\bigcup\limits_{\gamma\in S_{n+m}\atop 
	{st_{[n]}(\gamma(i)) =\sigma(i),i\leq n\atop 
	st_{[m]}(\gamma(i)) =\beta(i-n),i>n}}\Delta_{\gamma,t}^{n+m} 
	=\sum\limits_{\nu\in S_{n+m}\atop 
	{pr_n^+(\Delta_{\nu,t}^{n+m}) =\Delta_{\sigma,t}^n \atop 
	pr_m^-(\Delta_{\nu,t}^{n+m}) =\Delta_{\beta,t}^m}} \Delta_{\nu,t}^{n+m}.
$$
The so-called left and right half-shuffle products are then defined respectively by (summing up to the full shuffle product)

$$
	\Delta_{\sigma,t}^n < \Delta_{\beta,t}^m
	:=\bigcup\limits_{\gamma\in S_{n+m},
	\gamma^{-1}(n+m) =\sigma^{-1}(n)\atop 
	{st_{[n]}(\gamma(i))=\sigma(i),i\leq n\atop 
	st_{[m]}(\gamma(i)) =\beta(i-n),i>n}} \Delta_{\gamma,t}^{n+m},
$$
and
$$
	\Delta_{\sigma,t}^n > \Delta_{\beta,t}^m
	=\bigcup\limits_{\gamma\in S_{n+m},
	\gamma^{-1}(n+m) =n+\beta^{-1}(m)\atop 
	{st_{[n]}(\gamma(i)) =\sigma(i),i\leq n \atop 
	st_{[m]}(\gamma(i))=\beta(i-n),i>n}}\Delta_{\gamma,t}^{n+m} ,
$$ 
such that, on iterated integrals (with the same notations as in the previous section) the left and right half-shuffle products correspond respectively to the causal and anti-causal products. For example,
$$
	(X_\sigma \nwarrow X_\beta)(t)= \langle \Delta_{\sigma,t}^n <\Delta_{\beta,t}^m |H \rangle .
$$
It is easy to check that they define two commuting left and right actions of the algebra of simplices on itself or, equivalently, that the
three identities (\ref{EMids}) above still hold when considering simplices instead of time-dependent operators. For example, given $\sigma \in S_n$, $\beta \in S_m$ and $\nu \in S_p$, 
$$
	\Delta_{\sigma,t}^n<(\Delta_{\beta,t}^m\times \Delta_{\nu,t}^p)
	=\bigcup\limits_{\gamma\in S_{n+m+p},
	\gamma^{-1}(n+m+p) =\sigma^{-1}(n)\atop 
	{st_{[n]}(\gamma(i)) =\sigma(i),i\leq n\atop 
	{st_{[m]}(\gamma(i)) =\beta(i-n),n<i\leq n+m\atop 
	st_{[p]}(\gamma(i)) =\nu(i-n-m),i>n+m}}} 
	\Delta_{\gamma,t}^{n+m+p} 
	=(\Delta_{\sigma,t}^n<\Delta_{\beta,t}^m )< \Delta_{\nu,t}^p.
$$

In other terms, the algebra of time-dependent operators, with the binary maps $\nwarrow$ and $\nearrow$, and the algebra of simplices, with $<$ and $>$, are both examples of so-called shuffle algebras, where the latter notion is given in the next

\begin{defn} [Shuffle algebra] \label{def:shufflealgebra}
A shuffle algebra is a (non-unital) associative algebra $(A, \cdot_{\!\scriptscriptstyle A})$ whose product splits into the sum of two so-called half-shuffle products, $\cdot_{\!\scriptscriptstyle A} = \prec + \succ$, such that
\begin{itemize}
\item the two binary operations $\prec$ and $\succ$ define respectively a right and a left action of $A$ on itself, commuting with each other,
\end{itemize}
or, equivalently,
\begin{itemize}
\item the left and right half-shuffle products hold the following shuffle relations:
\begin{eqnarray}
	(M \prec N) \prec P&=&M \prec (N\cdot_{\!\scriptscriptstyle A} P), 	\label{shuffle1}\\
 	(M \succ N) \prec P&=&M \succ (N \prec P), 					\label{shuffle2}\\
	(M \cdot_{\!\scriptscriptstyle A} N) \succ P&=&M \succ (N \succ P).	\label{shuffle3}
\end{eqnarray}
\end{itemize}
\end{defn}

\begin{rmk}
When dealing with functions instead of operators (or with scalar-valued iterated integrals), then the algebra product $\cdot_{\!\scriptscriptstyle A}$ is commutative, that is $\int_0^tds\dot{M}(s)N(s)=\int_0^tdsN(s)\dot{M}(s)$. This amounts to the relation $(M\nwarrow N)(t)=(N \nearrow M)(t)$. The axioms simplify then to the sole relation $(M \nwarrow N) \nwarrow P=M \nwarrow (N \cdot_{\!\scriptscriptstyle A} P)$ that defines the notion of commutative shuffle algebras. See \cite{schutz} for details.
\end{rmk}

\begin{rmk}
A shuffle algebra $A$ can be augmented with a unit: set for $x \in A$, $1 \nwarrow x=x \nearrow 1:=0$, $x \nwarrow 1=1 \nearrow x:=x$. However the half products $1 \nwarrow 1$ and $1\nearrow 1$ cannot be defined consistently. The trick seems to originate in \cite{schutz} (in a commutative setting).
\end{rmk}

\begin{rmk}
A Theorem of Foissy \cite{Foissy} shows that the shuffle algebra of simplices is free as a shuffle algebra. The proof of the Theorem is rather involved and may be considered among the nicest achievements resulting from the theory of shuffle algebras. It is however already interesting to look at the subalgebra spanned by $\Delta_t^1$ \cite{lr}.
\end{rmk}

We set $a \succ b \prec c$ for $(a \succ b) \prec c=a \succ (b \prec c)$. Then we have the following

\begin{lem}
The free shuffle algebra $Sh(x)$ generated by a single letter $x$ has then a basis $B$ defined inductively by:
\begin{itemize}
\item $x\in B$
\item If $y,z\in B$, then $y\succ x \prec z$, $y \succ x$ and $x \prec z$ are in $B$.
\end{itemize}
\end{lem}

Indeed, by induction on the length of the elements (given by the number of $x$ appearing), the identities 
$$
	(y \succ x \prec z) \prec (y' \succ x \prec z') 
	= y \succ x \prec (z \cdot_{\!\scriptscriptstyle A} (y' \succ x \prec z'))
$$ 
and  
$$
	(y \succ x \prec z) \succ (y' \succ x \prec z')
	=((x \succ x \prec z)\cdot_{\!\scriptscriptstyle A}  y') \succ x \prec z')
$$ 
imply that the elements of $B$ span linearly $Sh(x)$. 

To prove that they are linearly independent, it is enough to check that, in the shuffle algebra spanned by simplices $\Delta^n_{\sigma,t}$, the elements constructed recursively when $\Delta_t^1$ is substituted for $x$ in the definition of $B$ contain each a simplex indexed by a permutation that does not appear in the expansion of the other terms. This follows recursively from the definition of the products $<$ and $>$ and the identities $p_n^+(\Delta_{\sigma,t}^n>\Delta_t^1<\Delta_{\beta,t}^m)=\Delta_{\sigma,t}^n$ and $p_m^-(\Delta_{\sigma,t}^n>\Delta_t^1<\Delta_{\beta,t}^m)=\Delta_{\beta,t}^m$, where $\sigma\in S_n$, $\beta\in S_m$. In particular, the shuffle subalgebra of the shuffle algebra of simplices generated by $\Delta_t^1$ is a free shuffle algebra.

Foissy's freeness properties imply that computations with iterated integrals are ``generic'' for shuffle calculus: an algebraic identity involving the symbols $\nwarrow$ and $\nearrow$, or $<$ and $>$, that holds for iterated integrals of an arbitrary time-dependent operator holds for arbitrary shuffle algebras.

\begin{rmk} 
The terminology {\rm{chronological product/algebra}} is due to Agrachev and Gamkrelidze \cite{AG} and is synonymous in their work with {\rm{pre-Lie product/algebra}}. The terminology {\rm{dendriform relations}} is often used for the three identities \eqref{shuffle1}-\eqref{shuffle3} encoding the splitting of the associativity of the shuffle product, and {\rm{dendriform algebra}} for the corresponding notion of algebra. We prefer, however, the terminology {\rm{shuffle relations/algebra}} for conceptual and historical reasons. Indeed, the notion of {\rm{shuffle}} is natural in the light of the idea of splitting a product into two, i.e., left and right, {\rm{half-shuffle}} products corresponding respectively to a right and left action. 
\end{rmk}


\subsection{Chronological calculus and pre-Lie products}
\label{ssect:chronoPreLie}

In this section we survey the Lie-theoretic perspective of chronological calculus, following Agrachev and Gamkrelidze \cite{AG}. Starting from the representa- tion-theoretic point of view, we consider a vector space  $A$ with a binary product $\triangleright \colon A \otimes A \to A$ and the associated bracket product $[a,b]_\triangleright := a \triangleright b-b \triangleright a$. Write $L_x$ for the linear endomorphism of $A$ defined by left multiplication:
$$
	L_x(y):=x\triangleright y,
$$
and define the usual commutator bracket of linear endomorphisms of $A$, $[L_x,L_y]:=L_x\circ L_y-L_y\circ L_x$.

\begin{defn} The pair $(A,\triangleright)$ is a left pre-Lie algebra if and only if for any $x,y\in A$, the identity $[L_x,L_y]=L_{[x,y]_\triangleright}$ holds, which is equivalent to the left pre-Lie relation   
$$
	x \triangleright (y\triangleright z) - (x\triangleright y)\triangleright z
	=y\triangleright (x\triangleright z) - (y\triangleright x)\triangleright z.
$$
\end{defn}

The canonical example of a pre-Lie algebra, which can be traced back to Caley's work on rooted trees, is given in terms of derivations. One verifies for example that the space of vector fields generated by the derivations $x^n\partial_x$ is a pre-Lie algebra with product $(x^n\partial_x)\vartriangleright(x^m\partial_x):=m \cdot x^{n+m-1}\partial_x$. We refer the reader to \cite{Cartier2,ChaLiv,Manchon2} for other examples and more details on the notion of pre-Lie algebra. 

\begin{rmk}
The notion of right pre-Lie algebra is defined similarly in terms of the product $\triangleleft \colon A \otimes A \to A$, together with the right pre-Lie relation $ x \triangleleft (y \triangleleft z) - (x \triangleleft y) \triangleleft z  = x \triangleleft (z \triangleleft y) - (x\triangleleft z) \triangleleft y$. We will abbreviate left pre-Lie to pre-Lie from now on.
\end{rmk}

One can always augment $A$ with a unit. The pre-Lie product on $A \oplus k$ is defined by $(a+l)\triangleright (b+l'):=a \triangleright b+l'a+lb+ll'$ for $a,b\in A,\ l,l'\in k$. Using the map $x \longmapsto L_x$, this observation allows to embed $(A, [ -, - ]_\triangleright)$ into the vector space of linear endomorphisms of $A\oplus k$ equipped with its canonical Lie bracket. In particular, there is a forgetful functor from pre-Lie to Lie algebras. In other terms

\begin{lem}\label{lem:Liepre-Lie}
The pair $(A,[ -, - ]_\triangleright)$ is a Lie algebra.
\end{lem}

The link with classical chronological calculus is as follows. As in the previous section, we denote by $(A,\cdot_{\!\scriptscriptstyle A})$ an algebra of time-dependent operators and consider the Lie bracket $[M,N](t)$. The integration by parts formula implies that 
$$
	[M,N](t)=\int_0^tdsM(s)\dot{N}(s)+\int_0^tds \dot{M}(s)N(s) 
				-\int_0^tdsN(s)\dot{M}(s) - \int_0^tds \dot{N}(s)M(s),
$$ 
which can be written as the difference of: 
$$
	(M\triangleright N)(t):=(M \nearrow N)(t) - (N \nwarrow M)(t) 
	=\int_0^tdsM(s)\dot{N}(s) - \int_0^tds\dot{N}(s)M(s)
$$
and $(N\triangleright M)(t)$
so that $[M,N](t)=[M,N]_\triangleright(t).$ That the algebra $A$ is indeed a pre-Lie algebra follows from the Jacobi identity:
\begin{align*}
	([M,N]_\triangleright\ \triangleright P)(t)
	&=([M,N]\triangleright P)(t)=\int_0^tds [[M(s),N(s)],\dot{P}(s)]\\
	&=-\int_0^tds[[N(s),\dot{P}(s)],M(s)]+\int_0^tds[[M(s),\dot{P}(s)],N(s)]\\
	&=(M\triangleright (N\triangleright P))(t) - (N\triangleright (M\triangleright P))(t).
\end{align*}

More generally, there is a forgetful functor from shuffle to pre-Lie algebras \cite{Aguiar}:

\begin{prop}\label{prop:shufpreLie}
Any shuffle algebra $(A,\prec,\succ)$ has the structure of a pre-Lie algebra with 
\begin{equation}
\label{preLieShuf}
	x \triangleright y := x \succ y- y \prec x.
\end{equation}
\end{prop} 

Note that $x \triangleleft y := x \prec y - y \succ x = - y \triangleright x$ is a right pre-Lie product. We omit the proof, that follows from the calculations above in the particular case of time-dependent operators.


\subsection{Time ordered products and enveloping algebras}
\label{ssect:timeordered}

One of the central ideas of chronological calculus is to develop a group theoretic approach adapted to the study of flows of differential equations, based on the finer structure provided by pre-Lie algebras rather than Lie algebras. A key step is the construction of enveloping algebras of pre-Lie algebras. 

Recall that the enveloping algebra $U(L)$ of a Lie algebra $L$ is an associative algebra (uniquely defined up to isomorphism) with the following properties: 
\begin{itemize}

\item the Lie algebra $L$ embeds in $U(L)$ (as a Lie algebra, where the Lie algebra structure on $U(L)$ is induced by the associative product, i.e., in terms of the commutator bracket $[x,y]:=xy-yx$)

\item for any associative algebra $A$ (we write $g_A$ for its associated Lie algebra, that is $A$ equipped with the commutator bracket), there is a -- natural -- bijection between the Lie algebra maps from $L$ to $g_A$ and the associative algebra maps from $U(L)$ to $A$.
\end{itemize}
Usually, the algebra $U(L)$ is constructed as the quotient of the free associative algebra over $L$ by the ideal generated by relations $[l,l']-ll'-l'l=0$, $l,l'\in L$. It is a cocommutative Hopf algebra with a coproduct determined by its action on the elements of $L$, i.e., for $l \in L$, $\Delta(l):= l \otimes \mathbf{1} + \mathbf{1} \otimes l$.

Since pre-Lie algebras are Lie algebras with extra structure, their enveloping algebras have also additional structures that have been described in the works of Oudom--Guin \cite{go} and Lada--Markl \cite{lm}. In the first part of this section, we briefly recall the results of these works. Recall first that the algebra of polynomials $k[V]$ over a vector space $V$ is equipped with a bi-commutative Hopf algebra structure such that given a monomial $v_1\cdots v_n$, $v_i\in V$, the coproduct reads
$$
	\Delta(v_1\cdots v_n)=\sum\limits_{I\coprod J=[n]}v_I\otimes v_J,
$$
where $v_I:=v_{i_1}\cdots v_{i_l}$, for $I=\{i_1,\dots,i_l\} \subset [n]:=\{1,\ldots,n\}$.

From now on $(V,\triangleright )$ denotes a pre-Lie algebra.

\begin{defn}\label{symetricbrace}
The brace map on $V$ 
$$
	k[V]\otimes V\longrightarrow  V
$$
$$
	P\otimes v\longmapsto  \{P\}v
$$
is defined inductively by
$$
	\{w\}v:=w \triangleright v,\ w \in V
$$
$$
	\{w_1,\dots,w_n\}v:=\{w_n\}(\{w_1,\dots,w_{n-1}\}v)-\sum\limits_{i=1}^{n-1}\{w_1,\dots,\{w_n\}w_i,\dots,w_{n-1}\}v.
$$
 \end{defn}

We define now a product $\ast$ on $k[V]$ in terms of the brace map: for $v_1,\dots,v_n$ and $w_1,\dots,w_m$ in $V$,
\begin{equation}\label{astproduct}
	(w_1 \cdots w_m) \ast (v_1 \cdots v_n)=\sum\limits_f W_0(\{W_1\}v_1)\cdots (\{W_n\}v_n),
\end{equation}
where the sum is over all maps $f$ from $\{1,\ldots,m\}$ to $\{0,\ldots,n\}$ and $W_i:=\prod_{j\in f^{-1}(i)} w_j$. For example, $w\ast v=wv+\{w\}v$. The central result of Oudom--Guin is the following

\begin{thm}
The product $\ast$ is associative and unital. It equips the vector space of polynomials $k[V]$ with the structure of an enveloping algebra of $V$ and makes $(k[V],\Delta)$ a Hopf algebra (the Hopf algebra structure being the one associated to the enveloping algebra structure of $k[V]$).
\end{thm}

Following \cite{EFP2014} we will provide an interpretation of the aforementioned result of Oudom--Guin in terms of chronological calculus. Note that the results in \cite{EFP2014} are stated in the more general framework of Rota--Baxter algebras. Assume from now on that the pre-Lie algebra $(V,\triangleright)$ is furthermore a shuffle algebra with the induced pre-Lie product \eqref{preLieShuf} -- for example, an algebra of time-dependent operators. Recall that the time-ordered product of two elements in $V$ is given by $T[v,w]:=v \prec w+w \prec v$. In the context of  time-dependent operators the product is obtained by symmetrising the causal product. In general, for $v_1,\ldots,v_n\in V$,
$$
	T[v_1,v_2,\dots ,v_n]:= \sum\limits_{\sigma\in S_n}
	v_{\sigma(1)} \prec (v_{\sigma(2)} \prec (\cdots  \prec (v_{\sigma(n-1)} \prec v_{\sigma(n)})\cdots )).
$$
The starting points for the following developments are:
\begin{itemize}

\item The observation that a pre-Lie algebra $V$ has in some sense more structure than a Lie algebra, which is reflected in the fact that its polynomial algebra, $k[V]$, is naturally equipped with two Hopf algebra structures (sharing the same underlying coproduct), i.e., the one associated to the associative product $\ast$ making $k[V]$ an enveloping algebra and the usual product of polynomials.

\item The fact that, since $V$ is a shuffle algebra, it is equipped with an associative product, denoted $\cdot$, whose bracket identifies with the Lie bracket $[ - , - ]_\triangleright$. By the universal properties of enveloping algebras, the identity map from $(V,[ - , - ]_\triangleright)$ to $(V,[-,-])$ induces a map of associative algebras, $\iota \colon (k[V]^+,\ast) \to (V,\cdot)$, where $k[V]^+$ stands for the augmentation ideal of $k[V]$ consisting of polynomials without constant terms.
\end{itemize}

In degree two we then have: 
\begin{align*}
	\iota(wv)
	=\iota(w\ast v)-\iota(\{w\}v)
	&=w\cdot v-w \triangleright v\\
	&=w \prec v+w \succ v - (w \succ v - v \prec w)\\
	&=w \prec v+v \prec w
	=T[w,v].
\end{align*} 
This is a general phenomenon (see \cite[p. 1291]{EFP2014} for a proof): 

\begin{thm} \label{thm:keyth}
The image in $V$ of a monomial $v_1 \cdots v_n$ in $k[V]$ by the canonical map $\iota$ is the time-ordered product of the $v_i$s:
 \begin{equation}
 \label{keyeq}
 	\iota (v_1\cdots v_n) = T[v_1,\dots,v_n].
 \end{equation}
\end{thm}


\subsection{Formal flows and Hopf algebraic structures}
\label{ssect:flows}

There are two classical ways to study the formal properties of the flow associated to a differential equation, such as for instance the evolution operator $X(t)$ in the notation of Theorem \ref{thm:bch}, when dealing with the linear differential equation $\dot{X}(t)=X(t)H(t)$. By formal, we mean in this section that questions related to convergence  of series are not addressed, i.e., infinite sums are treated as formal sums. 

The first approach to study flows relies on classical group theory and was described earlier in this article in terms of the discrete version of the Baker--Campbell--Hausdorff formula (Theorem \ref{thm:bch} above). It expresses the logarithm of the product of two exponentials and, as such, the formal link between a continuous group and its Lie algebra in the neighbourhood of the identity. The two are locally in bijection through the logarithm and exponential map and this bijection can be understood formally by studying the structure of the (suitably completed) enveloping algebras of free Lie algebras. For detail, see for instance Reutenauer's monograph \cite{Reutenauer}, where this approach is developed.

The other way to express the link between the solution $X(t)$ (of $\dot{X}(t)=X(t)H(t)$) and $\Omega(t):=\log(X(t))$ is given by the Magnus formula \cite{Magnus}, which describes the latter in terms of a highly non-linear differential equation  
$$
	\dot{\Omega}(t)
	=\frac{ad_{\Omega}}{e^{ad_{\Omega}}-1}H(t)
	=H(t)+\sum\limits_{n>0}\frac{B_n}{n!}ad_{\Omega(t)}^n(H(t)).
$$
Here $ad$ denotes the adjoint representation, where $ad_{M}^n(N):=[M,ad_{M}^{n-1}(N)]$, $ad_{M}^0(N):=N$, and the $B_n$ are the Bernoulli numbers (the numbers $\frac{B_n}{n!}$ are the coefficients of the expansion of $\frac{x}{\exp(x)-1}$ as a formal power series). Its pre-Lie structure was investigated in \cite{EFM2} and follows from the observation that, for time-dependent operators, the pre-Lie product 
$$
	(M\triangleright N)(t)=\int_0^tdu [M(u),\dot{N}(u)]
$$
satisfies $\frac{d}{dt}(M\triangleright N)(t)=ad_{M(t)}\dot{N}(t)$. The Magnus formula reads then:
\begin{equation}
\label{naiveMagnus}
	\dot{\Omega}(t)
	=\frac{d}{dt}\left( \big\lbrace\frac{\Omega}{\exp(\Omega) -1} \big\rbrace H\right)(t)
\end{equation}
where $\frac{\Omega}{\exp(\Omega)-1}$ is computed using the associative product of the enveloping algebra of the pre-Lie algebra of time-dependent operators. 

We explain in this section how this kind of phenomenon can be accounted for. Our presentation is based on the seminal work of Agrachev and Gamkrelidze \cite{AG} and later developments by Chapoton and Livernet \cite{ChaLiv}, Guin and Oudom \cite{go}, Ebrahimi-Fard and Manchon \cite{EFM2} and Chapoton and Patras \cite{ChapPat}.

Recall first from \cite{ChaLiv} and \cite{go} the following result.

\begin{thm}\label{chaliv}
Let $B=(v_n)_{n\in I}$ be a basis of a vector space $L$. Then, non-planar rooted trees with vertices indexed by elements of the set $I$ form a basis of $V:=\mathcal{PL}(L)$, the free pre-Lie algebra over $V$. A non-planar rooted tree is interpreted recursively as an element of $V$ as follows: if $T$ is the tree obtained by attaching the subtrees $T_1,\dots, T_n$ (viewed as elements of $L$) to a root vertex indexed by $v_i$, then 
$$
	T=\{T_1\cdots T_n\}v_i.
$$
The pre-Lie product on $V$ is obtained as follows: given two trees $T,T'$, the pre-Lie product $T\triangleright T'$ is the sum of the trees obtained by summing over all graftings of the root of $T$ successively onto every vertex of $T'$. The set of forests over non-planar rooted trees forms a basis of the polynomial algebra $k[V]$. 
\end{thm}

\begin{rmk} 
A forest $F$ being a set $\{T_1,\ldots ,T_k\}$ of non-planar rooted trees is interpreted as the monomial $T_1\cdots T_k$ in $k[V]$. As a Hopf algebra, $(k[V],*,\Delta)$ is known as the Grossman--Larson Hopf algebra \cite{GL}.
\end{rmk}

From now on, $V$ denotes a free pre-Lie algebra equipped with a basis of non-planar rooted trees, as in the previous Theorem. Its enveloping algebra is the associated Grossman--Larson Hopf algebra $k[V]$. It is graded by the number of vertices of the trees in a forest. We denote by $\widehat{k[V]}$ its completion with respect to this grading, which is compatible with the two products on $k[V]$, i.e., the commutative product of monomials and the associative product of the enveloping algebra. 

Recall the following general properties of the completion $\hat H$ of a graded connected cocommutative Hopf algebra $H$ \cite{mm}. The vector space of its primitive elements, denoted $Prim(\hat H):=\{x \in \hat H,\ \Delta(x) = x \otimes 1 + 1 \otimes x\}$, is a Lie algebra for the Lie bracket associated to the product in $\hat H$. The set $G(\hat H):=\{x \in \hat H,\ \Delta(x) = x\otimes x\}$ of group-like elements forms instead a group for the product in $\hat H$. The exponential and logarithm maps define two inverse bijections between $Prim(\hat H)$ and $G(\hat H)$. Notice that $\widehat{k[V]}$ is such a Hopf algebra with respect to its two natural Hopf algebra laws.

The vector space, $Prim(\widehat{ k[V]})=\hat V$ is the completion of $V$ as a graded vector space. It carries two Lie algebra laws. The one associated to the commutative product that we write from now on $\circ$, is trivial (making it a commutative Lie algebra), the other one is the Lie bracket associated to the pre-Lie product on $V$.

As a set, the group $G(\widehat{k[V]})=: G(V)$ consists of exponentials of elements in $\hat V$ (using any of the two products). We denote by $\exp^\circ$, $\log^\circ$ the isomorphism between $\hat V$ and $G(V)$ coming from the commutative product ($\circ$) and by $\exp^\ast$, $\log^\ast$ the isomorphism defined in terms of the associative product ($\ast$). 

\begin{rmk} 
Notice that $\hat V$ is equipped with two inverse set automorphisms: $\log^\circ \circ \exp^\ast$ and $\log^\ast \circ \exp^\circ$. We are going to show that the second encodes the Magnus formula. The first one identifies with the map $W$ in Agrachev and Gamkrelidze \cite{AG}.
\end{rmk}

\begin{defn} 
Let us call 2-group (resp.~2-pre-Lie algebra) a set equipped with two group (resp.~pre-Lie) structures. The 2-group (resp.~2-pre-Lie algebra) of formal flows associated to $V$ is the group $(G(V),\circ,\ast)$ of group-like elements in $\widehat{k[V]}$ (resp.~the triple $(\hat V,0,\triangleleft)$, where $0$ stands for the null pre-Lie product). 
\end{defn}

Let us focus here on the application of this formalism to the Magnus formula.

\begin{defn}
Let $x \in \hat V$. The Magnus element $\Omega(x) \in \hat V$ is the (necessarily unique) solution $\Omega (x)$ to the equation:
$$
	\big\lbrace{\frac{\Omega (x)}{\exp(\Omega (x)) -1}}\big\rbrace x =\Omega(x).
$$
\end{defn}

Note that the unicity follows by induction from grading arguments. 

\ \par 

The link with the classical Magnus formula is as follows. Consider $\int_0^t ds H(s)$ in an algebra $A$ of time-dependent operators (for example matrices $M(t)$ with $M(0)=0$ and with bounded $C^\infty$ coefficients, to avoid the discussion of convergence phenomena). Define then $V$ to be the free pre-Lie algebra over a variable $x$. There is a canonical pre-Lie algebra map from $V$ to $A$ induced by the map $x\longmapsto \int_0^tdsH(s)$ as well as an induced  map from the Grossman--Larson algebra $k[V]$ to $A$. These two maps send $\Omega(x)$ to $\Omega(t)$.

\ \par 

Notice that, by Theorem \ref{thm:keyth}, the last map also sends $\exp^\circ (x)-1$ to $X(t)-1$, where $X(t)$ is the solution of $\dot{X}(t)=X(t)H(t)$, $X(0)=1$. Since $\Omega(t)=\log(X(t))$, we get that the images in $A$ of the two elements $\Omega(x)$ and $\log^\ast\circ\exp^\circ(x)$ in $\widehat{k[V]}$ coincide. In fact, more generally

\begin{thm}
The Magnus element $\Omega(x)$ identifies with $\log^\ast \circ \exp^\circ(x).$
\end{thm}

\begin{proof}
Indeed,
we have:
$$
	\log^\ast(\exp^\circ(x))=\sum_{n>0}(-1)^{n-1}\frac{(\exp^\circ(x)-1)^{\ast n}}{n}.
$$
Let us write $u=v+o(1)$, meaning that $u$ and $v$ in $\widehat{k[V]}$ are equal up to a formal linear combination of forests in $\widehat{k[V]}$ with at least two trees. Denoting $\widehat{k[V]}_{\geq 2}$ the latter space, we have in particular $\widehat{k[V]} = k \oplus V \oplus \widehat{k[V]}_{\geq 2}$. Notice that, by definition of the $\ast$-product, 
$$
	\widehat{k[V]}\ast \widehat{k[V]}_{\geq 2} \subset \widehat{k[V]}_{\geq 2}.
$$
We then have: 
\begin{align*}
	\log^\ast(\exp^\circ(x))
	&= \left(\sum_{n>0}(-1)^{n-1}\frac{(\exp^\circ(x)-1)^{\ast n-1}}{n}\right)\ast \big(x+o(1)\big)\\
	&= \left(\sum_{n>0}(-1)^{n-1}\frac{(\exp^\circ(x)-1)^{\ast n-1}}{n}\right)\ast x+o(1).
\end{align*}
Recall now that for any $y \in \widehat{k[V]}$, $y \ast x= \{y\}x + o(1)$. Since $\log^\ast(\exp^\circ(x)) \in V$, and for $v,w\in V$, $(v=w+o(1))\implies (v=w)$, we obtain finally
$$
	\log^\ast(\exp^\circ(x))
	=\big\lbrace\sum_n(-1)^{n-1}\frac{(\exp^\circ(x)-1)^{\ast n-1}}{n}\big\rbrace x
	=\big\lbrace\frac{\log^\ast(\exp^\circ(x))}{\exp^\circ(x)-1}\big\rbrace x,
$$
from which the statement of the theorem follows.
\end{proof}

\newpage


\section{Rota--Baxter Algebras}
\label{sect:RBA}

The mathematician Gian-Carlo Rota once lamented in a 1998 paper \cite{Rota4} that: 
\begin{quote} 
``Whereas algebraists have devoted a lot of attention to derivations, the algebraic theory of the indefinite integral has been strangely neglected. The shuffle identities are only the tip of an iceberg of algebra and combinatorics of the indefinite integral operator which remains unexplored. The profound work of K.~T.~Chen bears witness to the importance of the algebra of the indefinite integral. J.~Milnor once told me that he considers K.~T.~Chen's work to be some of the most important and most neglected mathematics of the latter half of this century. However, not even Milnor's opinion has so far made a dent into the uncontrollable forces of fashion and fad''. (G.C.~Rota \cite{Rota4}) 
\end{quote}

Rota refers above to the unital associative shuffle algebra with its commutative shuffle product, which plays a critical role in Chen's seminal works \cite{Chen57,Chen71} (see also Ree's article \cite{Ree58}). As we know from the previous sections, the shuffle product encodes indeed in abstract algebraic terms the classical integration by parts rule and its extension to the product of iterated integrals. As we will see in this section, the second part of this survey, the notion of Rota--Baxter algebra defined through the Rota--Baxter relation provides a far reaching generalisation Chen's shuffle algebra.  

\smallskip

In the following we will outline the main algebraic aspects of the contemporary theory of Rota--Baxter algebras. We should remark though that we refrain from giving an exhaustive account. Let us emphasise the underlying aim. Rota--Baxter algebra is representative for a whole class of theories, including iterated integrals of scalar and operator valued functions -- corresponding to commutative and non-commutative shuffle products -- as well as summation and projection operators -- corresponding to commutative and non-commutative quasi-shuffle products. The main distinction between the commutative and the non-commutative realm is encoded in the presence of pre-Lie and post-Lie products. Furthermore, since the axioms of Rota--Baxter algebra encode more general ``integral-type operators'' than the Riemann integral of chronological calculus, the identities that are obtained in Rota--Baxter algebra have a much wider range of applications.


\subsection{Origin}
\label{ssect:Origin}

The Rota--Baxter identity and together with it the notion of (commutative) Rota--Baxter algebra first appeared in the 1960 work of the American mathematician Glen Baxter\footnote{Note that this Baxter should not be confused with the Australian theoretical physicist Rodney Baxter -- the latter is known, among others for the (quantum) Yang--Baxter equation.}~\cite{Baxter}. Baxter's paper was motivated by a fundamental result in probability theory due to Frank Spitzer \cite{Spitzer}, on which we will comment further below. Various proofs of the so-called Spitzer identity had been obtained before. Baxter's approach, however, which is based on the identity that bears his name, succeeded in unveiling the algebraic and combinatorial structures underlying Spitzer's identity.

About ten years later, Gian-Carlo Rota realised that Spitzer's identity can be deduced using classical results from the theory of symmetric functions \cite{Rota1,Rota2,Rota3}. Indeed, Rota showed that Spitzer's identity is \it equivalent \rm to the Waring identity. Since the algebra of symmetric functions is naturally equipped with a Hopf algebra structure, Rota's observation could have paved the way to the use of Hopf algebra techniques in the theory of Rota--Baxter algebras. However, this does not seem to have been the case: Hopf algebra structures were not put forward in the classical theory of Rota--Baxter algebras and their applications -- as far as we know. Here we will indicate briefly how some classical properties can be given Hopf-theoretic interpretations. We will also build a connection with the results in the first part of this article, where Hopf algebra structures emerged in relation to iterated integrals and Chen's rule.

Beside the links with the theory of symmetric functions, Rota and his school advocated the field of algebraic combinatorics, both mathematically and on epistemological grounds -- an attitude that can be recognised in his introduction to reference \cite{Rota1}:

\begin{quote} 
``The spectacular results in the fluctuation theory of sums of independent random variables [...] have gradually led to the realisation that the nature of the problem, as well as that of the methods of solution, is algebraic and combinatorial [...]. It is the present purpose to carry this algebraisation to the limit: the result we present amounts to a solution of the word problem for Baxter algebras''. (G.C.~Rota \cite{Rota1})
\end{quote}

Another essential component of the mathematical foundation of the theory of Rota--Baxter algebras can be traced back to Pierre Cartier's 1972 paper \cite{Cartier} on the structure of free commutative Rota--Baxter algebras. Similarly to Rota, Cartier considered the word problem, which amounts to constructing explicitly a basis of the free commutative Rota--Baxter algebra. However, he suggested a rather different approach which lead him to the notion of commutative quasi-shuffle product. The latter generalises Chen's classical shuffle product \cite{Chen57,Chen71} that was studied in the previous sections of the article. As we shall see, this extension from shuffles to quasi-shuffles is one of the (many) reasons to work in the axiomatic framework of Rota--Baxter algebra.

Cartier's algebra product appeared independently a few years later in a 1979 paper by Kenneth Newman and David E.~Radford \cite{NewRad}, where the authors endow the free coalgebra over an associative algebra with a Hopf algebra structure. A quarter of a century later, Jessica Gaines \cite{Gaines} introduced a recursive definition of the quasi-shuffle product -- similar to Cartier's -- in the context of the multiplication of iterated stochastic integrals in the case of Wiener processes. After Gaines, C.W.~Li  and X.Q.~Liu \cite{LiLiu} considered Gaines' quasi-shuffle product while studying the algebraic structure of multiple stochastic integrals with respect to Brownian motion and standard Poisson processes. Around the same time as Gaines' work, Robin Hudson and collaborators described a combinatorial formula for the product of iterated quantum stochastic integrals \cite{Beasley,HudPat}. 
This product was coined sticky-shuffle \cite{Hudson_2} and is equivalent to Gaines' quasi-shuffle product. In the light of the Hopf algebraic perspective alluded to earlier, we should also mention that Hudson described the sticky-shuffle (or quasi-shuffle) Hopf algebra in \cite{Hudson_1}. Eventually, and independently of the aforementioned developments, the notion of quasi-shuffle product was formalised by Michael Hoffman in \cite{Hoffman} by embedding it in a Hopf algebraic framework and detailing its relation with Chen's shuffle product. See also references \cite{foissy2,foissy3,HoffmanIhara16}. We remark in passing that Hoffman's main motivation came from understanding algebraic and combinatorial structures underlying multiple zeta values (MZVs) \cite{CartierMZV}. This rough (and likely incomplete) synopsis of the history of the commutative quasi-shuffle product suggests that Rota was right about the iceberg of algebra and combinatorics that comes with the notion of commutative shuffle product. Indeed, we will see that the latter is an abelian example of a special case of a larger structure known as non-commutative quasi-shuffle algebra.  

The central aim of this section of the survey will be indeed to advocate the non-commutative theory of Rota--Baxter algebras. In spite of multiple good reasons, theoretical as well as applied ones (e.g., integration and finite difference calculus with respect to operator algebra valued functions), interest in non-commutative Rota--Baxter algebra has been for long rather sporadic, compared to the amount of work that erupted shortly after Baxter's paper came out. This changed only in the last 15 years, starting largely with the three papers \cite{Aguiar,EF2002,GK} that (re)initiated interest for the algebraic structures and properties of (non-)commutative Rota--Baxter algebra. Since then, the theory has grown in both systematic developments as well as applications.

The momentum behind this interest stems initially from the use of non-commutative Rota--Baxter algebra in the context of renormalisation in perturbative quantum field theory. Recall that the renormalisation program permits to give sense to integrals appearing in perturbative computations in quantum field theory, which are otherwise divergent. The main idea is to identify and then extract and eliminate those parts that cause the divergencies in a coherent way, which is moreover compatible with the underlying physics \cite{CasKen,collins}. The extraction and elimination part, i.e., the renormalisation process, is combinatorial in nature. It was essentially understood and described by Bogoliubov and others (better known as Bogoliubov--Parasiuk--Hepp--Zimmermann (BPHZ) renormalisation method \cite{collins}), and has been known in physics for a long time\footnote{In fact, we remark that Bogoliubov's and collaborators' original article \cite{BoPa} appeared around the same time as Magnus' \cite{Magnus} and Spitzer's \cite{Spitzer} work}. See \cite{Bogo} for one of the earliest textbooks on the subject. The theoretical framework underlying renormalisation plays a fundamental role in the so-called standard model, and the spectacular results in particle physics that followed from it. Let us remark that from a mathematical point of view\footnote{Despite the fact that other approaches to the renormalisation problem, like, for instance, the Epstein--Glaser method \cite{EG} or Wilson's renormalisation group \cite{Wilson}, allow for both a comprehensive understanding of the origin of divergencies as well as the very nature of the renormalisation program.} Bogoliubov's subtraction scheme still carries a veil of mystery.              

In the late 1990s, Connes and Kreimer came up with a rather unexpected approach to the BPHZ renormalisation method in terms of a combinatorial Hopf algebra of Feynman graphs \cite{CK1,Kreimer}. In this approach, the recursive subtraction algorithm of Bogoliubov et al.~is considered in the context of dimensional regularisation together with minimal subtraction. It was understood that the subtraction scheme map has to satisfy a particular set of identities; it was pointed out by C.~Brouder that these identities could be subsumed into a single relation, which was henceforth called {\it{multiplicativity constraint}} \cite{KreimerC}. A few years later it was realised by the first author of this article, that the identity proposed by Brouder coincides with the Rota--Baxter identity (of weight $\theta=1$). Subsequent works clarified the role of the Rota--Baxter relation in the aforementioned work of Connes and Kreimer and allowed to relate their results to structural properties of Rota--Baxter algebras such as Atkinson's decomposition \cite{Atkinson}. We refer the reader interested in details to Manchon's article in \it La Gazette des math\'ematiciens \rm\cite{Manchon1}, where a concise and elegant account of the algebraic structures involved in the Connes--Kreimer theory, including the Rota--Baxter algebra aspect, is presented.      

For what considers this presentation, two aspects are of crucial importance. On the one hand, Bogoliubov's recursive algorithm and the algebraic Birkhoff decomposition derived by Connes and Kreimer are an instance of a general result, known as Atkinson's factorisation, which is valid in any Rota--Baxter algebra. On the other hand, the work of Connes and Kreimer triggered considerable research into what is nowadays known as the theory of combinatorial Hopf algebras as well as into the theory of pre-Lie algebras. This lead also to a renewed interest in non-commutative Rota--Baxter algebras.
 
\smallskip
 
The rest of this work is organised as follows. In the following section we recall the definition of Rota--Baxter algebra together with related algebraic structures of Lie theoretic type. Then Atkinson's factorisation for Rota--Baxter algebras is presented, followed by the classical, i.e., commutative Spitzer and Bohnenblust--Spitzer identities. In subsection \ref{ssect:ncCartier-Rota} we describe briefly Cartier's and Rota's constructions of free commutative Rota--Baxter algebras. The generalisation of Spitzer-type formulas to non-commutative Rota--Baxter algebras is presented in subsection \ref{ssect:ncSpitzer}. In this context, the notion of pre-Lie algebra is central. We indicate how for non-commutative Rota--Baxter algebras the theory of non-commutative symmetric functions plays a role analogue to that of symmetric functions for commutative Rota--Baxter algebras described by Rota.


\subsection{Definition and examples}
\label{ssect:Defalg}

\begin{defn}\label{def:RBA} 
A {\rm{Rota--Baxter algebra of weight $\theta \in k$}} consists of an associative $k$-algebra $A$ equipped with a linear operator $R \colon A \to A$ satisfying the {\rm{Rota--Baxter relation of weight $\theta$}}:
\begin{equation}
\label{RBR}
	R(x)R(y)=R\big(R(x)y+xR(y) - \theta xy \big) \qquad \forall x,y \in A.
\end{equation}
\end{defn}
A Rota--Baxter algebra is said to be commutative if it is commutative as an associative algebra. Note that if $R$ is a Rota--Baxter map of weight $\theta$, then the map $R':=\beta R$ for $\beta \in k$ different from zero is of weight $\beta \theta$. This permits to rescale the original weight  $\theta \neq 0$ to the new $\theta' = -1$ (or $\theta' = +1$). Rota--Baxter maps come in pairs, that is, the linear map 
$$
	\tilde{R}:= \theta Id - R
$$ 
satisfies as well the Rota--Baxter relation of weight $\theta$ \eqref{RBR}. Here $Id$ denotes the identity map on $A$. 

\begin{lem}
Let $(A,R)$ be a Rota--Baxter algebra of weight $\theta$, the pair $(A,\tilde{R})$ is also a weight $\theta$ Rota--Baxter algebra.
\end{lem}

We leave it to the reader to check that 
\begin{equation}
\label{mixedRBR}
	R(x)\tilde{R}(y) = R\big(x \tilde{R}(y)\big) + \tilde{R}\big(R(x)y\big)
	\quad \mathrm{and} \quad
	\tilde{R}(x)R(y) = \tilde{R}\big(x R(y)\big) + R\big(\tilde{R}(x)y\big).
\end{equation}	

\begin{rmk}
Rota--Baxter algebras do not necessarily have a unit as associative algebras. When they don't, we augment them with a unit as follows. Let $B$ be a Rota--Baxter algebra over a field $k$, we write $\overline B$ for the augmented unital algebra $k\oplus B$ (with augmentation the projection onto $k$ orthogonally to $B$) and call it the augmented Rota--Baxter algebra (of $B$). However, we \it do not \rm extend in general the operator $R$ to $\overline B$, so that $R$ is then defined only on $B$.
\end{rmk}

We consider first some examples by starting with the case when the weight $\theta$ equals zero and the algebra $A$ is commutative. The most natural interpretation of a commutative weight zero Rota--Baxter map is given in terms of the integration by parts relation for the indefinite Riemann integral. Indeed, take for example $A$ to be the algebra $C^\infty(\RR)$ of smooth functions on the real line and set $R(f)(t):=\int_0^t f(x)dx$. For two functions $f,g \in C^\infty(\RR)$ we obtain the weight zero Rota--Baxter identity:
\begin{equation}
\label{RBR0}
	R(f)R(g)=R\big(R(f)g+fR(g)\big).
\end{equation}
Even in this somewhat degenerate case, the theory is not without interest since it embraces, among other aspects, Chen's commutative iterated integral calculus, Sch\"utzenberger's theory of commutative shuffle algebras \cite{schutz} as well as a large part of Fliess' approach to control theory \cite{Fliess} (related to chronological algebras in the sense of \cite{Kawski}). Indeed, for $R$ being the integral operator, we recognise in $R(f_1R(f_2 \cdots R(f_{n}) \cdots ))$ the iterated integral of functions $f_1, \ldots ,f_n$. 

Recall that the shuffle product $(f_1,\dots ,f_n)\shuffle (g_1,\dots ,g_m)$ of two \it sequences \rm $(f_1,\ldots,f_n)$ and $(g_1,\dots ,g_m)$ is the multiset of sequences $(h_1, \ldots ,h_{m+n})$ made of the letters $f_i$ and $g_j$ where the partial orders of the $f_i$ and of the $g_j$ are preserved. For example, $(f_1,g_1,f_2,g_2,f_3,g_3,f_4)$ or $(f_1,f_2,g_1,g_2,f_3,f_4,g_3)$ are among the shuffles of $(f_1,f_2,f_3,f_4)$ and $(g_1,g_2,g_3)$. However, $(f_1,f_4,f_2,g_1,g_2,f_3,g_3)$ is not, since $f_4$ appears on the left of $f_2$ and $f_3$. For example,  
\begin{align*}
	(f_1,f_2)\shuffle(g_1,g_2)
	&=(f_1,f_2,g_1,g_2)+(f_1,g_1,f_2,g_2)+(f_1,g_1,g_2,f_2)\\
	& \quad +(g_1,g_2,f_1,f_2)+(g_1,f_1,g_2,f_2)+(g_1,f_1,f_2,g_2).
\end{align*}
When letters are repeated, sequences are counted with multiplicities: 
$$
	(f,g)\shuffle (f)=\{(f,f,g), (f,f,g),(f,g,f)\}.
$$

From \eqref{RBR0} follows 
\begin{lem}
In a weight zero Rota--Baxter algebra, we have:
\begin{eqnarray}
\label{shuffle}
	\lefteqn{R\big(f_1R(f_2 \cdots R(f_{n}) \cdots )\big)R\big(g_1R(g_2 \cdots R(g_m) \cdots )\big)} \nonumber\\
	&\hspace{6cm} =\sum\limits_{(h_1, \ldots ,h_{m+n})\in\atop (f_1,\dots ,f_n)\shuffle (g_1,\dots ,g_m)
	} R\big(h_1R(h_2 \cdots R(h_{n+m}) \cdots )\big).
\end{eqnarray}
\end{lem}
Relation (\ref{shuffle}) yields in particular the well known shuffle product of two iterated integrals due to Chen. 

The case of a non-commutative weight zero Rota--Baxter algebra includes the algebras of chronological calculus (algebras of time-dependent operators stable under a suitable integral map satisfying the integration by parts formula). We will return to this particular case in a specific subsection below, making explicit the links between the formalism of Rota--Baxter algebras and chronological calculus \`a la Agrachev--Gamkrelidze.

\smallskip

Next we should address the obvious question whether there are interesting examples of non-zero weight Rota--Baxter maps? From the point of view of analysis and for the theory of discrete dynamical systems, the fundamental example of non-zero weight Rota--Baxter maps are summation operators. 

\begin{exam}[Finite summation operators]
On functions $f$ defined on $\mathbb{N}$ and with values in an associative algebra $A$, the summation operator $R(f)(n):=\sum_{k=0}^{n-1}f(k)$ is a Rota--Baxter map of weight $\theta = -1$. It is the right inverse of the finite difference operator $\Delta(f)(n):=f(n+1) - f(n)$.       
\end{exam}

Another example traces the Rota--Baxter relation back to its origin in classical fluctuation theory. The latter deals with extrema of sequences of real valued random variables. Their distribution can be studied using operators on random variables such as $X\to \max(0,X)$. This motivates the

\begin{exam}[Fluctuation theory]
The operator $R$ is defined on the characteristic function $F(t):=\EE[\exp({itX})]$ of a real valued random variable $X$:
$$
 	R(F)(t):=\EE[\exp({itX^+})],
$$
where $X^+:=\max(0,X)$. One can show that $R$ is indeed a Rota--Baxter map of weight $\theta = 1$. This example motivated Spitzer's and hence Baxter's original works \cite{Baxter,KuRoYa,Rota1,Spitzer}. We will return to it further below. 
\end{exam}     

\begin{exam}[Split algebras]
Another, more algebraic, example derives from an associative algebra $A$ (not necessarily commutative) which decomposes into a direct sum of two (non-unital) subalgebras, $A=A^+\oplus A^-$. The two orthogonal projectors $\pi_\pm: A \to A^\pm$ are Rota--Baxter maps of weight $\theta=1$. In fact, any projector which satisfies relation (\ref{RBR}) is of weight $\theta=1$. A rather well known example for such an algebra are Laurent series $\CC[\epsilon^{-1},\epsilon]]$, which decompose into a ``divergent'' part $\epsilon^{-1}\CC[\epsilon^{-1}]$ and a ``regular'' part $\CC[[\epsilon]]$. Further below we will comment on its role in the modern approach to renormalisation in perturbative quantum field theory due to Alain Connes and Dirk Kreimer \cite{CK1}. 
\end{exam}

\begin{exam}[Matrix algebras]
An interesting class of examples is provided by $n \times n$ matrices together with certain projectors. We denote by $\pi_+$ the projector that maps a matrix $\alpha$ to the upper triangular matrix $\alpha^+$ defined by replacing all entries in $\alpha$ below the main diagonal by zeros. The projector $\pi_-:= Id -\pi_+$ maps $\alpha$ to $\pi_-(\alpha)=\alpha^-$, the corresponding strictly lower triangular matrix. It is not hard to see that $\pi_\pm$ satisfy \eqref{RBR} for weight $\theta=1$. Note that the projector defined by mapping $\alpha$ to the difference $\alpha^- - (\alpha^-)^\dagger$, where $(\alpha^-)^\dagger$ is transpose of $\alpha^-$, satisfies the forthcoming relation \eqref{RBRLie}, but not \eqref{RBR}. It therefore defines a Lie Rota--Baxter algebra structure (to be defined below) on the $n \times n$ matrices -- corresponding to the $QR$-factorisation of matrices.         
\end{exam}

Various of these examples are commutative Rota--Baxter algebras with a non-zero weight. It is interesting to see, as a warm up for the forthcoming developments, how the classical product formulas for iterated integrals transform to take into account the weight.

The product formula in \eqref{shuffle} holds indeed only in the commutative weight zero case: it is already wrong when computing $R(x)R(y)$. Indeed, looking closer at the weight $\theta$ Rota--Baxter relation \eqref{RBR} reveals that it has a shuffle part \eqref{shuffle} and an additional $\theta$-weight part. In the following example we separate the parts coming from ``pure'' shuffling from the additional components that result from the weight $\theta$ term using brackets to identify the latter.        
$$
	R(x)R(y)=R(xR(y))+R(yR(x)) - [\theta R(xy)], 
$$
\begin{align*}
	R(x)R(yR(z)) 	&= R(xR(yR(z)))+R(yR(xR(z)))+R(yR(zR(x)))\\
				& \qquad\ - [\theta R(xyR(z))] + [\theta^2 R(yR(xz))].
\end{align*}
Formulae for larger products follow from a recursive application of \eqref{RBR}. A closed formula follows from the weight-zero \eqref{shuffle} one by replacing shuffle permutations by quasi-shuffle surjections. A $(m,n;r)$-quasi-shuffle of type $\max(m,n) \le r < m+n$ is a surjection
$$
	\sigma:\{1,\ldots, m+n\}\surj\{1,\ldots, r\},
$$
such that  $\sigma(1) < \cdots < \sigma(m)$ and $\sigma(m+1) < \cdots < \sigma(m+n)$. We denote by $\mop{Sh}(m,n;r)$ the set of $(m,n;r)$-quasi-shuffles of type $\max(m,n) \le r \le m+n$. 
\begin{lem}
The product of iterated Rota--Baxter maps in a commutative weight $\theta$ Rota--Baxter algebra $A$ is given by
\begin{align}
\label{surjectionshuffle}
	& R(x_1R(\cdots R(x_m)\cdots)) R(x_{{m+1}}R(\cdots R(x_{{m+n}})\cdots))\\ 
	& \qquad\quad\ = \sum_{r=\max(m,n)}^{m+n}\,(-\theta)^{m+n-r} 
	\sum_{\sigma\in \mathrm{Sh}(m,n;r)} R(y_{1}^\sigma R( \cdots R(y_{{r}}^\sigma)\cdots)) \nonumber
\end{align}
with $y_i^\sigma:=\prod_{j\in \sigma^{-1}(\{i\})}x_i$.
\end{lem}
 
Notice that the set $\sigma^{-1}(\{i\})$ contains one or two elements. As in the pure shuffle case, these identities are valid only in the commutative case.


\subsection{Related algebraic structures}
\label{ssect:algstru}

\noindent What sets Rota--Baxter algebras apart are the rich algebraic structures induced by the relation \eqref{RBR}. In the weight zero case, we will show that one recovers the structures familiar in chronological calculus. In general, however, new structures show up involving quasi-shuffles.

It is natural and useful to extend the notion of Rota--Baxter operator to more general situations that the commutative or associative case:

\begin{defn}
Let $A$ be a vector space equipped with a bilinear product denoted $m_\centerdot: A \times A \to A$. A linear operator on $A$ is called a Rota--Baxter operator of weight $\theta$ if and only if, for any $x,y$ in $A$
\begin{equation}
\label{generalRB}
	m_\centerdot (R(x), R(y))= R(m_\centerdot (x, R(y)) +m_\centerdot (R(x), y) - \theta m_\centerdot (x, y)).
\end{equation}
When the product defines a given algebraic structure (say of $\mathcal P$-algebra) on $A$ (for example, a Lie algebra), then $(A,m_\centerdot,R)$ is called a Rota--Baxter $\mathcal P$-algebra (a Rota--Baxter Lie algebra, for example).
\end{defn}

Notice that with these conventions, a Rota--Baxter algebra (resp.~a commutative Rota--Baxter algebra) as we defined it is equivalently a Rota--Baxter associative algebra (resp.~a Rota--Baxter commutative and associative algebra). These shortcuts should cause no harm in practice.

\begin{exam}
Let $(A,R)$ be a Rota--Baxter algebra. One verifies quickly that on the Lie algebra $L_A$, defined on $A$ in terms of the usual commutator bracket, the map $R$ satisfies for all elements $x,y \in A$ the Lie analog of the Rota--Baxter relation of weight $\theta$:
\begin{equation}
\label{RBRLie}
	[R(x),R(y)]=R\big([R(x),y] + [x,R(y)] - \theta [x,y] \big),
\end{equation}
making $(L_A,[-,-],R)$ a Rota--Baxter Lie algebra of weight $\theta$. In other terms, the forgetful functor from associative to Lie algebras (given by the construction of the bracket as an anti-commutator) goes over to a forgetful functor from associative Rota--Baxter to Rota--Baxter Lie algebras. These phenomena have been studied in general in \cite{Bai}.
\end{exam}

For an associative Rota--Baxter algebra $A$ of weight $\theta$ the operator $B:=R - \tilde{R}$ satisfies the modified Rota--Baxter identity of weight $\theta$:
\begin{equation}
\label{modRBR}
	B(x)B(y)=B(B(x)y+xB(y)) - \theta^2 xy.
\end{equation}
On the level of Lie algebra the corresponding Lie analog of \eqref{modRBR} is
\begin{equation}
\label{ybm}
	[B(x),B(y)]=B([B(x),y]+[x,B(y)]) - \theta^2 [x,y].
\end{equation}

\begin{rmk}\label{rmk:mCYBE}
Relations \eqref{RBRLie} and \eqref{ybm} are well-known in the context of classical integrable systems, where they are called modified classical Yang--Baxter equations \cite{STS}. Semenov-Tian-Shansky observed that \eqref{modRBR} defines an associative analog of the modified classical Yang--Baxter equation.
\end{rmk}

\smallskip 

We now introduce one of the key properties of Rota--Baxter algebras. 

\begin{defn}[Rota--Baxter product]
The argument of the map $R$ on the right hand side of \eqref{RBR} consists of a sum of three terms; one can show that it defines a new associative product on $A$. We denote this new product by $m_{\ast_\theta}$ and call it the Rota--Baxter associative product:
\begin{equation}
\label{doubleasso}
	m_{\ast_\theta}(x , y)=x \ast_\theta y := R(x)y + xR(y) - \theta xy.
\end{equation}
\end{defn}

It is obvious that \eqref{RBR} implies that $R(x \ast_\theta y )=R(x)R(y)$ and $\tilde{R}(x \ast_\theta  y )=-\tilde{R}(x)\tilde{R}(y)$.  In particular
\begin{lem}
The Rota--Baxter map $R$ is a morphism of algebras from $(A,\ast_\theta)$ to $(A,\cdot)$.
\end{lem}

We remark that in terms of the modified Rota--Baxter map $B:=R - \tilde{R}$ we find that $x \ast_\theta y = \frac{1}{2}(B(x)y + xB(y))$ and $\llbracket x , y\rrbracket_\theta = \frac{1}{2}([B(x),y]+[x,B(y)])$.

\begin{exam}[Commutative shuffles] Consider a weight zero commutative Rota--Baxter algebra and its Rota--Baxter product $\ast_0=\ast$ defined in \eqref{doubleasso}. One can look at it as a splitting of an associative and commutative product into two components (the so-called left and right half-shuffle products):
$$
	f \ast g = f\succ g + f\prec g 
$$
where 
\begin{equation}
\label{halfSh}
	f\prec g:=fR(g),\quad f\succ g:=R(f)g. 
\end{equation}
The two abstract products, $\prec$ and $\succ$, satisfy the axiomatic characterisation of commutative shuffle algebras that appeared independently in algebraic topology and in Lie theory in the 1950's \cite{EM,schutz}:
\begin{equation}
\label{demishuffle0}
	a \prec b=b\succ a,\quad\   (a\prec b)\prec c=a\prec (b\prec c+c\prec b).
\end{equation}
These two relations are actually enough to insure that the product defined by the sum $\prec + \succ$ is associative as well as commutative \cite{schutz}. Notice that knowing one of the two products $\prec,\succ$ is enough since they determine each other.
Sch\"utzenberger computed also the free commutative shuffle algebra over an alphabet $X$ (the algebra generated by $X$ and a product $\prec$ such that setting $x\succ y:=y\prec x$, the two products satisfy the above relations). He showed that this algebra is (up to isomorphism) the tensor algebra over $X$ where, using a word notation for tensor products:
$$
	x_1\cdots x_n \prec y_1\cdots y_m:=x_1(x_2\cdots x_n\prec y_1\cdots y_m
									+y_1\cdots y_m\prec x_2\cdots x_n),
$$
so that the product $\ast$ identifies with the usual shuffle product of words.
In control theory this algebra would (sometimes) be called the free chronological algebra over $X$ (in the sense of Kawski whose terminology differs from that used by Agrachev--Gamkrelidze): this is one of the (many) relations between Rota--Baxter algebras and chronological calculus. In the theory of operads, this algebra is very often called the free Zinbiel algebra over $X$ \cite{lodV}, a tribute to Cuvier's dual notion of Leibniz algebras \cite{cuv1,cuv3}.
\end{exam}

\begin{exam}[Commutative quasi--shuffles]
Let us now assume that the weight $\theta=1$. The idea is similar to the weight zero case (i.e.~commutative shuffles). In the product (\ref{doubleasso}) we then have three terms:
$$
	x \ast_1 y = R(x)y + xR(y) - xy.
$$
The first two terms correspond to the aforementioned ``half-shuffles''. The last part, $xy$, is just the original commutative and associative algebra product. Using the notations $x \prec y:=xR(y)$, $x\succ y:=R(x)y$ and $x \cdot y:=xy$ yields:
\begin{equation}
\label{demi-quasi-shuffle}
	x\succ y=y\prec x,\quad\   (x\prec y)\prec z=x\prec (y\prec z+z\prec y - y \cdot z).
\end{equation}
This is the commutative quasi-shuffle product \cite{Hoffman}. When compared to the ordinary shuffle product (of weight zero), the last term, $y \cdot z$ may be considered a deformation of the shuffle product. Recall that the quasi-shuffle product is very present in the context of the theory of multiple zeta values (MZVs) and other generalisations of special functions (see \cite{CartierMZV}). 
\end{exam}

\begin{exam}
In a weight zero non-commutative Rota--Baxter algebra $A$ we do no longer have $a \prec b = b \succ a$, but since 
\begin{align*}
	aR(b\ast_0 c)	&=a(R(b)R(c))=(aR(b))R(c),\\
	(R(a)b)R(c)	&=R(a)bR(c)=R(a)(bR(c)),\\
	R(a)(R(b)c)	&=R(a)R(b)c=(R(a)R(b))c=R(a\ast_0 b)c,
\end{align*}
the maps $a\longmapsto a\prec b:=aR(b)$ and $a\longmapsto a\succ b:=R(a)b$ define two commuting right and left representations of the algebra $(A,\ast_0)$ on itself. That is, the half-shuffles satisfy the system of axioms that characterises the notion of non-commutative shuffles \cite{Aguiar,Manchon1} (see also Definition \ref{def:shufflealgebra}):

\begin{prop}[Non-commutative shuffles] \label{thm:shufIto} In a weight zero Rota--Baxter algebra $A$, setting $a \prec b:=aR(b)$, $a\succ b:=R(a)b$, the following three shuffle relations are satisfied:
\begin{equation}
\begin{aligned}
\label{shuffleNC}
	(a\prec b)\prec c&=a\prec (b\prec c+b\succ c)\\
	a\succ (b\prec c)&=(a\succ b)\prec c\\
	a\succ (b\succ c)&=(a\prec b+a\succ b)\succ c,
\end{aligned}
\end{equation}
making $A$ a shuffle algebra. The shuffle product $\shuffle=\succ+\prec$ is the Rota--Baxter product, i.e., $\shuffle=\ast_0$.
\end{prop}
\end{exam}

\begin{remark}\label{QSprod}
In a weight $\theta=1$ non-commutative Rota--Baxter algebra $A$ we still have 
\begin{align*}
	aR(b\ast_1 c)	&=a(R(b)R(c))=(aR(b))R(c),\\
	(R(a)b)R(c)	&=R(a)bR(c)=R(a)(bR(c)),\\
	R(a)(R(b)c)	&=R(a)R(b)c=(R(a)R(b))c=R(a\ast_1 b)c,
\end{align*}
and the maps $a\longmapsto aR(b)$, $a\longmapsto R(b)a$ define two commuting right and left representations of the algebra $(A,\ast_1)$ on itself. However, the two half-shuffles $a\prec b:=aR(b)$ and $a\succ b:=R(a)b$ do not add up to the $\ast_1$ product, so that $A$ is not a shuffle algebra.
\end{remark}

\smallskip

The previous identities imply that the half-shuffles satisfy the system of axioms that characterises non-commutative quasi-shuffle algebras (also called tridendriform algebras in the literature). These relations seem to have appeared first in stochastic analysis, i.e., in the work of Karandikar on It\^o calculus for semimartingales \cite{Karandikar1982a} (see also \cite{cur}).

\begin{thm}[Non-commutative quasi-shuffles]\label{thm:quasishufIto} 
Setting $\ast:=\ast_1$ for the Rota--Baxter product, and writing $a \cdot b:=ab$ for the usual associative product in $A$ (so that $\ast:=\prec +\succ -\ \cdot$), the left- and right half-shuffles, $a\prec b:=aR(b)$ and $a\succ b:=R(a)b$, satisfy the following identities defining a (non-commutative) quasi-shuffle algebra
\begin{equation}
\label{quasishuffleNC}
\begin{tabular}{ l l }
	$(a \prec b) \prec c = a \prec (b\ast c)$, 		&\quad $(a \succ b) \cdot c = a \succ (b \cdot c)$ \\  
	$a \succ (b \succ c) = (a\ast b) \succ c$, 		&\quad  $(a \prec b) \cdot c = a \cdot (b \succ c)$ \\
	$(a \succ b) \prec c = a \succ (b \prec c)$, 	&\quad  $(a \cdot b) \prec c = a \cdot (b \prec c)$. 
\end{tabular}
\end{equation}
\end{thm}

\begin{prop}\cite{EF2002}\label{prop:link} Let $(A,R)$ be a Rota--Baxter algebra of weight $\theta=1$. Define $a \prec b:=aR(b) - ab = -a \tilde{R}(b)$ and $a \succ b:=R(a)b$. Then $(A, \prec, \succ)$ satisfies the shuffle algebra axioms \eqref{shuffleNC}.
\end{prop}

\medskip

One checks easily that $R$ is still a Rota--Baxter operator for the new product $\ast_\theta$, that is, the Rota--Baxter relation still holds when one replaces the initial product on $A$ by the product $\ast_\theta$:

\begin{lem}
Let $(A,\cdot,R)$ be a Rota--Baxter algebra of weight $\theta$, the algebra $(A,\ast_\theta,R)$ is again a weight $\theta$ Rota--Baxter algebra, denoted $A^\theta$.
\end{lem}

Moreover, in light of \eqref{RBRLie} one verifies that 

\begin{lem}The product
\begin{equation}
\label{doubleLie}
	\llbracket x , y\rrbracket_\theta := [R(x),y] + [x,R(y)] - \theta [x,y]
\end{equation}
defines another Lie bracket on the Lie algebra $(L_A, [-,-])$. The corresponding Lie algebra is denoted $({L}_A^\theta,\llbracket-,-\rrbracket_\theta)$. 
\end{lem}

Analogously to the associative case, we have that $R$ becomes a Lie morphism from ${L}^\theta_A$ to $L_A$, i.e., $R(\llbracket x , y\rrbracket_\theta)=[R(x),R(y)]$. 

 \begin{rmk}
Hence, a Rota--Baxter algebra is a vector space with two associative products (resp.~Lie brackets) related by  \eqref{doubleasso} (resp.~\eqref{doubleLie}) and morphisms of algebras. This leads to another, equivalent, characterisation of Rota--Baxter algebras since the fact that $R$ is a map of algebras from $A^\theta$ to $A$ is precisely the Rota--Baxter relation, from which properties of the two Lie brackets follow.
\end{rmk}

\smallskip

Equations \eqref{doubleasso} and \eqref{doubleLie} relate two associative products respectively two Lie brackets. Saying this, one may wonder whether the terms on the righthand side of both identities, i.e., $R(x)y$, $xR(y)$ respectively $[R(x),y]$, $[x,R(y)]$, possesses more properties. Regarding the former, we have seen in Theorem \ref{thm:shufIto} and Theorem \ref{thm:quasishufIto} that they define shuffle (weight zero), respectively quasi-shuffle algebras (non-zero weight). Let us now consider the Lie algebra case \eqref{doubleLie}, that is, $\theta \neq 0$. It turns out that for $x,y \in A$ the bilinear product 
$$
	x \rhd y:=[R(x),y]
$$ 
does not satisfy any relation, meaning that it is magmatic in nature. However, in \cite{Bai} it was shown that on the Lie algebra $L_A(\theta)$ with Lie bracket $[x,y]^\theta :=-\theta[x,y]$ the following relations are satisfied
\begin{align}
	x \rhd [ y,z ]^\theta &= [x \rhd y,z]^\theta 
					+ [ y, x \rhd z ]^\theta 		\label{postLie1} \\
	[ x,y ]^\theta  \rhd z &= x \rhd (y \rhd z ) - (x \rhd y) \rhd z  
					- y \rhd (x \rhd z ) + (y \rhd x) \rhd z.  	\label{postLie2}
\end{align}
Identity \eqref{postLie1} follows immediately from the Jacobi identity. A quick computation yields \eqref{postLie2}:
\begin{align*} 
	[ x,y ]^\theta  \rhd z 
	&= -\theta[R([x,y]),z]\\
	&= -[R([R(x),y]),z]-[R([x,R(y)]),z] + [[R(x),R(y)],z]\\
	&=  [R(x),[R(y),z]]  -[R([R(x),y]),z]  - [R(y),[R(x),z]] + [R([R(y),x]),z]\\
	&= x \rhd (y \rhd z ) - (x \rhd y) \rhd z   
					- y \rhd (x \rhd z ) + (y \rhd x) \rhd z. 			
\end{align*}

Relations \eqref{postLie1} and \eqref{postLie2} yield what is called a post-Lie algebra structure on $L_A(\theta)$.

\begin{defn}[Post-Lie algebra]\label{def:postLie}
Let $(\mathfrak{g},[-,-])$ be a Lie algebra. A {\rm{post-Lie algebra}} structure on $\mathfrak{g}$ is given in terms of a triple $(\mathfrak{g},[-,-],\rhd )$ with the operator  $\rhd : \mathfrak{g} \otimes \mathfrak{g} \to \mathfrak{g}$ satisfying the post-Lie relations 
\begin{align}
	x \rhd [ y,z ] &= [x \rhd y,z]
					+ [ y, x \rhd z ]		\label{postLie11} \\
	[ x,y ]  \rhd z &= x \rhd (y \rhd z ) - (x \rhd y) \rhd z  
					- y \rhd (x \rhd z) + (y \rhd x) \rhd z.  	\label{postLie22}
\end{align}\end{defn}
For any post-Lie algebra one can show that the bracket
$$
	\llbracket x , y\rrbracket := x \rhd y - y \rhd x - [x,y]
$$  
defines a new Lie bracket on $\mathfrak{g}$ and the corresponding Lie algebra is denoted $(\mathfrak{g},\llbracket - , -\rrbracket)$. Hence, in a (Lie) Rota--Baxter algebra, identity \eqref{doubleLie} is an example of such a new Lie bracket. We refer the reader to \cite{EFM2018} and references in there for more details on the notion of post-Lie algebras and its relevance in other fields, e.g., geometric numerical integration on manifolds and classical integrable systems.

\begin{cor}[Pre-Lie algebra]\label{cor:pre-Lie}
In case of an abelian Lie algebra, the post-Lie algebra $(\mathfrak{g},[-,-],\rhd )$ reduces to a left pre-Lie algebra $(\mathfrak{g},\rhd)$ with left pre-Lie identity:
\begin{equation}
\label{pre-Lie}
	0=(x \rhd y) \rhd z - x \rhd (y \rhd z)  
					- (y \rhd x) \rhd z + y \rhd (x \rhd z).
\end{equation}
As any pre-Lie algebra is a Lie algebra, the bracket $\llbracket x , y\rrbracket := x \rhd y - y \rhd x$ defines a Lie algebra $(\mathfrak{g},\llbracket - , -\rrbracket)$. 
\end{cor}

Note that in the weight $\theta=0$ case, we have that the Lie bracket $[-,-]^\theta=0$. 

\begin{cor}
The product $x \rhd y:=[R(x),y]$ defines a left pre-Lie algebra on any weight zero Rota--Baxter algebra $(A,R)$. The pre-Lie relation \eqref{pre-Lie} implies that the commutator
$$
	[x , y]_0 := x \rhd y - y \rhd x
$$
defines a Lie bracket on $A$. 
\end{cor}

In general, any non-commutative Rota--Baxter algebra comes therefore naturally with a post-Lie algebra structure. It reduces to a pre-Lie algebra one in the case of zero weight. However, the following result shows that the picture is more involved. Returning to Proposition \ref{prop:link}, we observe

\begin{prop}\label{prop:pre-LieRB}
Let $(A,R)$ be an associative Rota--Baxter algebra of weight $\theta$. For all $x,y \in A$ the product
\begin{equation}
\label{pre-LieNew}
	x \bullet_\theta y := R(x)y + y\tilde{R}(x)
\end{equation}
satisfies the left pre-Lie identity \eqref{pre-Lie}.
\end{prop} 

We leave the proof as a simple exercise. Note that $x \bullet_\theta  y =  [R(x),y] + \theta yx$ and that this pre-Lie product gives the Lie bracket in \eqref{doubleLie}. In fact, both left and right pre-Lie products can be defined, since $x\, {}_\theta\!\!\bullet\! y:=- y\!\bullet_\theta\! x$ is right pre-Lie. For $\theta = 0$, we have $x \!\bullet_0\! y = [R(x),y]:=R(x)y-yR(x)$. Obviously, in the commutative case the pre-Lie product in \eqref{pre-LieNew} simplifies to the original product of the algebra (up to a sign and scaling by the weight $\theta$). Also, note that in a proper Lie Rota--Baxter algebra $(\mathfrak{g},R)$ the statement of Proposition \ref{prop:pre-LieRB} is not available.


\subsection{Atkinson's factorisation and Bogoliubov's recursion}
\label{ssect:Atkinson}

We turn now to typical problems in the theory of Rota--Baxter algebras, closely related to the behaviour of time-ordered exponentials -- recall that when dealing with weight zero Rota--Baxter algebras, the latter are the sums of iterated integrals that we investigated at the beginning of this article.


Let $A$ be a unital Rota--Baxter algebra of weight $\theta$. Both Atkinson \cite{Atkinson} and Baxter \cite{Baxter} studied the following fixpoint equations for $x \in A$:
\begin{equation}
\label{atkinson}
	\harpoonl{\exp}_{{\!\scriptscriptstyle{R}}}(\lambda x)
	= 1+ \lambda R\big(\harpoonl{\exp}_{{\!\scriptscriptstyle{R}}}(\lambda x)\, x\big), 
	\qquad\  
	\harpoonr{\exp}_{\tilde{{\!\scriptscriptstyle{R}}}}(\lambda x) 
	= 1+ \lambda \tilde{R}\big(x\, \harpoonr{\exp}_{\tilde{{\!\scriptscriptstyle{R}}}}(\lambda x)\big).
\end{equation}
Notice that setting $\lambda =1$, taking for $R$ the indefinite Riemann integral (a weight zero Rota--Baxter operator), and replacing $x$ by $H(t)$, a time-dependent matrix (with, for example, bounded $C^\infty$ coefficients), one gets by a perturbative expansion for $\harpoonl{\exp}_{{\!\scriptscriptstyle{R}}}(\lambda x)$ the time-ordered exponential associated to the differential equation
$$
	\dot{M}(t)=M(t)H(t).
$$
The formal parameter $\lambda$ is introduced to circumvent any discussions regarding convergence or invertibility issues. Identity \eqref{mixedRBR} implies the central result in $A[[\lambda]]$:
$$
	\harpoonl{\exp}_{\!\scriptscriptstyle{R}}(\lambda x) 
	\harpoonr{\exp}_{\tilde{\!\scriptscriptstyle{R}}}(\lambda x) 
	=1 + \lambda\theta\, \harpoonl{\exp}_{\!\scriptscriptstyle{R}}(\lambda x)\, x\, \harpoonr{\exp}_{\tilde{\!\scriptscriptstyle{R}}}(\lambda x), 
$$  
which we check directly
\begin{align*}
	\lefteqn{\harpoonl{\exp}_{\!\scriptscriptstyle{R}}(\lambda x)
	\harpoonr{\exp}_{\tilde{{\!\scriptscriptstyle{R}}}}(\lambda x) =
	\big(1+ \lambda R\big(\harpoonl{\exp}_{\!\scriptscriptstyle{R}}(\lambda x)\, x\big)\big)
	\big(1 + \lambda \tilde{R}\big(x\, \harpoonr{\exp}_{\tilde{{\!\scriptscriptstyle{R}}}}(\lambda x)\big)\big)} \\
	&= 1 + \lambda R\big(\harpoonl{\exp}_{\!\scriptscriptstyle{R}}(\lambda x)\, x\big) 
		+ \lambda \tilde{R}\big(x\, \harpoonr{\exp}_{\tilde{{\!\scriptscriptstyle{R}}}}(\lambda x)\big) 
		 + \lambda^2 R\big(\harpoonl{\exp}_{\!\scriptscriptstyle{R}}(\lambda x)\, x\big)
		 \tilde{R}\big(x\, \harpoonr{\exp}_{\tilde{{\!\scriptscriptstyle{R}}}}(\lambda x)\big) \\
	&= 1 + \lambda R\big(\harpoonl{\exp}_{\!\scriptscriptstyle{R}}(\lambda x)\, x\big) 
		+ \lambda \tilde{R}\big(x\, \harpoonr{\exp}_{\tilde{{\!\scriptscriptstyle{R}}}}(\lambda x)\big) \\
	&\qquad	
	+ \lambda^2 R\big(\harpoonl{\exp}_{\!\scriptscriptstyle{R}}(\lambda x)\, 
	x\, \tilde{R}\big(x\, \harpoonr{\exp}_{\tilde{{\!\scriptscriptstyle{R}}}}(\lambda x)\big)\big)
	+ \lambda^2 \tilde{R}\big(R\big(\harpoonl{\exp}_{\!\scriptscriptstyle{R}}(\lambda x)\, x\big)\,
	x\, \harpoonr{\exp}_{\tilde{{\!\scriptscriptstyle{R}}}}(\lambda x)\big)\\
	&= 1 + \lambda R\big(\harpoonl{\exp}_{\!\scriptscriptstyle{R}}(\lambda x)\, 
	x\, \harpoonr{\exp}_{\tilde{{\!\scriptscriptstyle{R}}}}(\lambda x)\big) 
		+ \lambda \tilde{R}\big(\harpoonl{\exp}_{\!\scriptscriptstyle{R}}(\lambda x)\, 
		x\, \harpoonr{\exp}_{\tilde{{\!\scriptscriptstyle{R}}}}(\lambda x)\big) \\	
	&= 1 + \lambda\theta \, \harpoonl{\exp}_{\!\scriptscriptstyle{R}}(\lambda x)\, 
	x\, \harpoonr{\exp}_{\tilde{{\!\scriptscriptstyle{R}}}}(\lambda x).	
\end{align*}
This then yields Atkinson's factorisation:
\begin{equation}
\label{Atkins}
	1-\lambda \theta x = \harpoonl{\exp}_{\!\scriptscriptstyle{R}}(\lambda x)^{-1}
	\harpoonr{\exp}_{\tilde{{\!\scriptscriptstyle{R}}}}(\lambda x)^{-1}.
\end{equation}
Moreover, note that the inverses are given through the following formulas:
$$
	\harpoonl{\exp}_{\!\scriptscriptstyle{R}}(\lambda x)^{-1}
	=1 -  \lambda R\big(x\, \harpoonr{\exp}_{\tilde{{\!\scriptscriptstyle{R}}}}(\lambda x)\big)
$$ 
and 
$$
	\harpoonr{\exp}_{\tilde{{\!\scriptscriptstyle{R}}}}(\lambda x) ^{-1}
	=1 -  \lambda\tilde{R}\big(\harpoonl{\exp}_{\!\scriptscriptstyle{R}}(\lambda x) \, x\big).
$$ 
Let us check the first statement:
\begin{align*} 
	&\harpoonl{\exp}_{{\!\scriptscriptstyle{R}}}(\lambda x)
	\big(1 - \lambda R\big(x\, \harpoonr{\exp}_{\tilde{{\!\scriptscriptstyle{R}}}}(\lambda x)\big)\big)
	=1 + \lambda R\big(\harpoonl{\exp}_{{\!\scriptscriptstyle{R}}}(\lambda x)\, x\big)
		- \lambda R\big(x\, \harpoonr{\exp}_{\tilde{{\!\scriptscriptstyle{R}}}}(\lambda x)\big)\\
	&\hspace{6cm}	-\lambda^2 R\big(\harpoonl{\exp}_{\!\scriptscriptstyle{R}}(\lambda x)\, x\big)
		R\big(x\, \harpoonr{\exp}_{\tilde{{\!\scriptscriptstyle{R}}}}(\lambda x)\big)\\
	&=1 + \lambda R\big(\harpoonl{\exp}_{{\!\scriptscriptstyle{R}}}(\lambda x)\, x\big)
		- \lambda R\big(x\, \harpoonr{\exp}_{\tilde{{\!\scriptscriptstyle{R}}}}(\lambda x)\big)
		-\lambda^2 R\big(\harpoonl{\exp}_{\!\scriptscriptstyle{R}}(\lambda x)\, x \,
		R\big(x\, \harpoonr{\exp}_{\tilde{{\!\scriptscriptstyle{R}}}}(\lambda x)\big)\\
	&\qquad	-\lambda^2 R\big(R\big(\harpoonl{\exp}_{\!\scriptscriptstyle{R}}(\lambda x)\, x\big)
		\, x\, \harpoonr{\exp}_{\tilde{{\!\scriptscriptstyle{R}}}}(\lambda x)\big)
		+ \lambda^2 \theta R\big(\harpoonl{\exp}_{\!\scriptscriptstyle{R}}(\lambda x)
		\, x^2\, \harpoonr{\exp}_{\tilde{{\!\scriptscriptstyle{R}}}}(\lambda x)\big)\\
	&=1 + \lambda R\big(\harpoonl{\exp}_{{\!\scriptscriptstyle{R}}}(\lambda x)\, x\big)
		-\lambda R\big(\harpoonl{\exp}_{\!\scriptscriptstyle{R}}(\lambda x)
		\, x\, \harpoonr{\exp}_{\tilde{{\!\scriptscriptstyle{R}}}}(\lambda x)\big)		
		+\lambda^2 R\big(\harpoonl{\exp}_{\!\scriptscriptstyle{R}}(\lambda x)\, x \,
		\tilde{R}\big(x\, \harpoonr{\exp}_{\tilde{{\!\scriptscriptstyle{R}}}}(\lambda x)\big)\\
	&=1   -\lambda R\big(\harpoonl{\exp}_{\!\scriptscriptstyle{R}}(\lambda x)
		\, x\, \harpoonr{\exp}_{\tilde{{\!\scriptscriptstyle{R}}}}(\lambda x)\big)		
		+\lambda R\big(\harpoonl{\exp}_{\!\scriptscriptstyle{R}}(\lambda x)\, x \,
		\harpoonr{\exp}_{\tilde{{\!\scriptscriptstyle{R}}}}(\lambda x)\big)\\
	&=1.
\end{align*}
The other identities follow by similar computations

\smallskip

Bogoliubov's recursion was deviced completely independently of the theory of Rota--Baxter algebras as a way to compute the so-called counterterm and renormalised amplitudes in perturbative quantum field theory (pQFT). From the work of Connes and Kreimer it follows that it can be rephrased in terms of Rota--Baxter algebras -- see e.g. \cite{EFGP} for an account of the relations between the theory of renormalisation and Rota--Baxter algebra, Hopf algebras as well as group and Lie-theoretical structures.

Let us assume, to fit with pQFT, that $\theta=1$ and that $R$ is an idempotent operator ($R^2=R$) such that furthermore $R(1)=0$. We also assume that the Rota--Baxter algebra is graded as an algebra, with $R$ of degree $0$ as a linear endomorphism and $x=x_1+\dots+x_n+\cdots$ and $x_i$ of degree $i$. The standard example of such a situation, relevant for renormalisation, is when $A$ is an algebra of linear forms on a graded connected coassociative coalgebra with values in Laurent series and $R$ is the projection onto the regular part (the product of linear forms being the convolution product induced by the coalgebra coproduct).

We also omit the parameter $\lambda$ and set $f:=\harpoonl{\exp}_{\!\scriptscriptstyle{R}}( x),\ h:=\harpoonr{\exp}_{\!\scriptscriptstyle{\tilde R}}( x)$.
Bogoliubov's formulas compute the inverse $h^{-1}$ as well as $f$ (the latter is called ``counterterm'' in renormalisation). We sketch it briefly. Notice first that from $R^2=R$ we get $-R\tilde R=R(R - Id)=0=(R- Id)R=-\tilde R R$. So that from $f=1+R(fx)$ we obtain $\tilde R(f)=1$. From $f(1-x)=h^{-1},$ we get :
$$
	h^{-1}=-fx+1+R(fx)
	=1-\tilde R(fx).
$$
\begin{lem}[Bogoliubov recursion]
The two resulting coupled equations
$$
	h^{-1}=1-\tilde{R}(fx),\quad\ f=1 + R(fx)
$$
can then be solved recursively since, for $n>0$,
$$
	h^{-1}_n=-\tilde R(x_n)
	-\sum_{0<i<n}\tilde{R}(f_ix_{n-i}),\quad\ f_n=R(x_n)+\sum_{0<i<n}  R(f_ix_{n-i}),
$$
where we denote with an index $i$ the $i$-th term of each series.
\end{lem}
 
In low degrees we get $h^{-1}_1=-\tilde{R}(x_1)$, $f_1= R(x_1)$, $h^{-1}_2=-\tilde{R}(f_1x_1+x_2)$, $f_2=R(f_1x_1+x_2)$, and so on.


\subsection{Spitzer's Identity: commutative case}
\label{sect:Spitzer}

The interest in Rota--Baxter algebras results from the possibility to describe interesting and useful universal combinatorial identities that hold in a wide variety of contexts. The most famous example of such an universal combinatorial result is known as Spitzer's classical identity \cite{Spitzer}. 

We have seen that the indefinite Riemann integral operator
$$
	R(f)(t):=\int_0^tf(s)ds,
$$ 
is a natural weight zero Rota--Baxter map, say, on scalar-valued functions. Recall that the solution to the integral fixpoint equation 
$$
	y(t)= 1 + R(fy)(t)
$$ 
is given in terms of the exponential function 
$$
	y(t)=\exp\big(R(f)(t)\big).
$$ 
Going back to Atkinson's factorisation  in terms of the fixpoint equations in \eqref{atkinson}, one may wonder how the simple exponential solution in the particular weight zero case generalises to Rota--Baxter maps of arbitrary weight? Baxter \cite{Baxter} understood that Spitzer's classical identity \cite{Spitzer} provides the answer in the case of commutative Rota--Baxter algebra of arbitrary weight. 

Spitzer's original result provides a way to compute the characteristic function of the extrema of a discrete process $S_i:=X_1+ \cdots +X_i$, defined as a sequence of partial sums of a sequence of independent and identically distributed real valued random variables $X_i$. More precisely, considering the new sequence of random variables $Y_i:=\sup(0,S_1,S_2, \ldots ,S_i)$, Spitzer's formula permits to calculate the characteristic function of $Y_i$ in terms of  the positive part of the partial sums, denoted $S_j^+,\ j\leq i$. Let $F:=\EE[\exp(itX)]$ be the characteristic function of $X:=X_1$ and $R(F):=\EE[\exp(itX^+)]$, where $X^+:=Sup(0,X)$ (and similarly for other random variables as well as on the algebra generated by their characteristic functions). In the setting of Rota--Baxter algebra one then finds:
$$
	\EE[\exp({itY_k})]=R(F\cdot R(F\cdot\  \cdots  R(F \cdot R(F)) \cdots ))=:R^{(k)}(F),
$$
and:
$$
	\EE[\exp({itS_k^+})]=R(F^k).
$$ 

From this one deduces a formal identity which is true in every commutative Rota--Baxter algebra of any weight~$\theta$, namely:
\begin{prop}[Spitzer identity]\label{Spitzer} In a weight $\theta$ commutative Rota--Baxter algebra, we have (formally)
\begin{eqnarray}
\label{SpitzId}
	\log\Big(1+\sum\limits_{n>0}R^{(m)}(F)\Big)
					= R\Big(\sum\limits_{n>0}\frac{\theta^{n-1}F^n}{n}\Big).
\end{eqnarray}
\end{prop}

Observe that the argument of the Rota--Baxter map $R$ on the right hand side of \eqref{SpitzId} is a logarithm
$$
	\Omega_\theta'(F):=\sum\limits_{n>0}\frac{\theta^{n-1}F^n}{n}=-\theta^{-1}\log(1 - \theta F).
$$
Since $F=\theta^{-1}(\exp({\theta\Omega_\theta'(F)})-1)$,
one can express $\Omega_\theta'(F)$ in terms of the generating function of the Bernoulli numbers $\frac{x}{\mathrm{e}^x-1}$: 
\begin{equation}
\label{magnus}
	\Omega_\theta'(F) 
	= \frac{\ell_{\theta \Omega_\theta'(F)}}{\mathrm{e}^{\ell_{\theta \Omega_\theta'(F)}}-1}(F)
	= F + \sum_{n > 0} \frac{B_n}{n!} \ell_{-\theta \Omega_\theta'(F)}^{n}(F).
\end{equation}
Here $\ell$ is the left multiplication operator, $\ell_{x}(y):=xy$, and $B_n$ are the Bernoulli numbers. 

Going back to \eqref{atkinson} and using that $1+\sum\limits_{n>0}R^{(m)}(\lambda x)$ is the perturbative expansion of the solution of the fixpoint equation $\harpoonl{\exp}_{{\!\scriptscriptstyle{R}}}(\lambda x)=1+\lambda R(\harpoonl{\exp}_{{\!\scriptscriptstyle{R}}}(\lambda x)x)$, this result implies for a general commutative Rota--Baxter algebra $(A,R)$ that 
$$
	\harpoonl{\exp}_{{\!\scriptscriptstyle{R}}}(\lambda x)
	= \exp\big(R(\Omega_\theta'(\lambda x))\big) 
	\qquad\  
	\harpoonl{\exp}_{{\!\scriptscriptstyle{R}}}(\lambda x)^{-1}
	= \exp\big(-R(\Omega_\theta'(\lambda x))\big),  
$$
and similarly for $\harpoonr{\exp}_{\tilde{\!\scriptscriptstyle{R}}}(\lambda x)$. In the commutative case, Atkinson's factorisation then becomes an obvious statement
\begin{align}
\label{Atkins2}
	1-\lambda \theta x 
	&= \exp\big(-R(\Omega_\theta'(\lambda x))\big)\exp\big(-\tilde{R}(\Omega_\theta'(\lambda x))\big)\\
	&= \exp\big(-R(\Omega_\theta'(\lambda x)) -\tilde{R}(\Omega_\theta'(\lambda x))\big)\nonumber\\
	&= \exp\big(-\theta \Omega_\theta'(\lambda x))\big)\nonumber\\
	&= \exp\big(\log(1 - \lambda \theta x)\big).
\end{align}

The expression \eqref{magnus} will remind the reader of Magnus' classical formula for the logarithm of the fundamental solution of a first order matrix or operator differential equation \eqref{naiveMagnus}. See \cite{BCOR08} for a comprehensive review on the Magnus expansion and applications. Although the expression appears to be somewhat pointless in the commutative case, it will allow later for a straightforward generalisation of the above identity to non-commutative Rota--Baxter algebras. 

\smallskip

Deeper into the combinatorial nature of commutative Rota--Baxter algebras goes the 

\begin{prop}[Bohnenblust--Spitzer formula]
In a commutative weight $\theta$ Rota--Baxter algebra, we have the following identity
\begin{equation}
\label{clBSp}
	\sum_{\sigma \in S_n}\!\! R\bigl(\cdots R(F_{\sigma(1)})F_{\sigma(2)} \cdots \bigr)F_{\sigma(n)} 
			= \sum_{\pi \in \mathcal{P}_n}\! \theta^{n -|\pi|}\! \sideset{}{^{*_{\theta}}}
							\prod_{\pi_i \in \pi}(b_i-1)! \, \Bigl(\prod_{j \in \pi_i}F_j\Bigr),   
\end{equation}
where $S_n$ is the symmetric group of order $n$. On the right hand side, $\mathcal{P}_n$ denotes set partitions of $[n]:=\{1,\ldots,n\}$, and $|\pi|$ is the number of blocks of the partition $\pi \in \mathcal{P}_n$. The size of the $i$th block $\pi_i \in \pi$ is denoted $b_i$.
\end{prop}

Note that the product $\sideset{}{^{*_{\theta}}}\prod$ is defined in terms of the Rota--Baxter product in (\ref{doubleasso}). For instance
$$
	\sum_{\sigma \in S_2} R(F_{\sigma(1)})F_{\sigma(2)} 
	= R(F_{1})F_{2} + R(F_{2})F_{1} = \theta F_{1}F_{2} + F_{1} *_{\theta} F_{2},
$$
which is the weight $\theta$ Rota--Baxter identity up to applying $R$ on both sides. In the weight zero case, i.e., for $\theta=0$ the Bohnenblust--Spitzer formula simplifies to the obvious identity: 
$$
	\sum_{\sigma \in S_n} R\bigl(\cdots R( R(F_{\sigma(1)})F_{\sigma(2)})\cdots \bigr) F_{\sigma(n)}
	= \sideset{}{ ^{*_{0}}}\prod^n_{i=1} F_i.
$$

With the goal to generalise identity (\ref{clBSp}) to non-commutative Rota--Baxter algebras, we again propose a slightly more involved rewriting. Recall the canonical cycle decomposition $c_1 \cdots c_k$ of a permutation $\sigma \in S_n$. Each cycle is written starting with its maximal element, and the cycles are written in increasing order of their first entries. For example, $(32)(541)(6)(87)$ is such canonical cycle decomposition. The $j$th cycle is denoted $c_j=(a_{j_0} a_{j_1}\cdots a_{j_{m_j-1}})$, where $m_j$ is the size of this cycle and $a_{j_0}>a_l$, ${j_1} \le l \le j_{m_j-1}$.    

Identity (\ref{clBSp}) can be then restated as follows:
\begin{equation}
\label{clBSpPerm}
	\sum_{\sigma \in S_n} R\bigl(\cdots R( R(F_{\sigma(1)})F_{\sigma(2)})\dots \bigr)F_{\sigma(n)} 
	=	\sum_{\sigma \in S_n}  \mathcal{D}^\theta_{\sigma}(F_1,\ldots,F_n), 
\end{equation}  
where for each permutation $\sigma$ written in its canonical cycle decomposition $c_1 \cdots c_k$ we define:
\begin{equation}
\label{lr}
	\mathcal{D}^\theta_{\sigma}(F_1,\ldots,F_n) 
	:= \sideset{}{^{*_{\theta}}}\prod_{j=1}^{k}\Big( \big(\sideset{}{^{\circ}}\prod_{i=1}^{m_j-1}r_{\theta F_{a_{j_i}}}\big)(F_{a_{j_0}}) \Big).
\end{equation}   
Now the operator $r$ denotes right multiplication, that is, $r_{x}(y):=yx$. Note that the second product on the right hand side is with respect to the composition of these multiplication operators. For instance, the permutation $\sigma :=(43)(512) \in S_5$ yields the term:
$$
	\mathcal{D}^\theta_{\sigma}(F_1,\ldots,F_5) 
	= r_{\theta F_{3}}(F_{4}) *_\theta (r_{\theta F_{2}}r_{\theta F_{1}})(F_{5})  
	= \theta^3 (F_4 F_3)*_{\theta}(F_5 F_1F_2).
$$  
In the commutative setting, the product $(\prod_{i=1}^{m_j-1}r_{\theta F_{a_{j_i}}})(F_{a_{j_0}}) $ is independent of the order of the $a_{j_i},\ i\geq 1$. Therefore, several of the terms $\mathcal{D}^\theta_{\sigma}(F_1,\ldots,F_n)$ on the right hand side of (\ref{clBSpPerm}) may coincide. The resulting coefficients $\prod_{j=1}^{k}(m_j-1)!$ lead to the compact form (\ref{clBSp}) written in terms of set partitions. Indeed, choose a block $\pi_j$, say, of size $m_j$, of a partition $\pi$. Its first entry is fixed to be its maximal element. This block then corresponds to $(m_j-1)!$ different cycles of size $m_j$.


\subsection{Free commutative Rota--Baxter algebras}
\label{ssect:freecRB}

The construction of free Rota--Baxter algebras was obtained independently by Cartier and Rota.
One of the main motivations underlying their works, and especially the one of Rota, was to deduce universal identities such as Spitzer's. Their two approaches are rather complementary. Cartier's construction \cite{Cartier} of a free commutative Rota--Baxter algebra is a forerunner of the modern notion of the quasi-shuffle algebra (Example \ref{QSprod}) and provides an abstract inductive way to obtain \eqref{surjectionshuffle}. Cartier's construction was made more explicit in \cite{GK}. 


We feature here Rota's approach \cite{Rota1,Rota2}, that has naturally a Hopf algebraic flavor. It is based on the observation that for any commutative algebra $A$, the algebra $A^{\mathbf N}$ of functions from the non-negative integers into $A$, that is, sequences of elements of $A$ with pointwise product, is naturally equipped with a weight $\theta=-1$ Rota--Baxter algebra structure defined by the operator:      
$$
	R(a_1,\ldots ,a_n,\ldots ) = (0,a_1,a_1+a_2,\ldots ,a_1+\cdots +a_n,\ldots ).
$$
Taking $A$ to be the algebra $k[x_1,\ldots ,x_n,\ldots ]$ of polynomials in an infinite number of indeterminates, one can show that the Rota--Baxter subalgebra generated by the sequence $x:=(x_1,\ldots ,x_n,\ldots )$ is free. This establishes an immediate link to the well-understood classical theory of symmetric functions. In fact, terms such as $R(x^n)$ and $R^{(n)}(x):=R(R^{(n-1)}(x)x), \ R^{(1)}(x):=R(x), \ R^{(0)}(x):=1$, correspond in the theory of symmetric functions to power sums  ($h_n:=\sum_ix_i^n$) and elementary symmetric functions ($e_n:=\sum_{i_1<\cdots <i_n}x_{i_1} \cdots x_{i_n}$), respectively. From this, Rota was able to  conclude that Spitzer's identity is equivalent to the Waring formula, which expresses power sums as polynomials in elementary symmetric functions.

The algebra of symmetric functions $\mathcal S$ is a free commutative algebra over the $e_n$ (or the $h_n$).  Rota's construction shows that there is a canonical embedding of $\mathcal S$ into $RB(x)$, the free Rota--Baxter on one generator. His embedding maps $e_n$ to $R^{(n)}(x)$ and therefore $RB(x)$ contains as a free commutative subalgebra the algebra freely generated by the $R^{(n)}(x)$. Moreover, $\mathcal S$ carries naturally a Hopf algebra structure for which the power sums are primitive ($\Delta(h_n)=h_n\otimes 1+1\otimes h_n$) or, equivalently, the sum of all elementary symmetric functions is group-like (as $\Delta(e_n)=\sum_{i=0}^ne_i\otimes e_{n-i}$). This Hopf algebra structure has many nice and important properties surveyed in Macdonald's \cite{macdo} -- for example, it is inherited from induction and restriction among Young subgroups of symmetric groups. This motivates the

\begin{defn}
The commutative Spitzer algebra is the bicommutative Hopf algebra that, as a commutative algebra, is the free commutative algebra generated by the $R^{(n)}(x)$ in $RB(x)$. The coproduct is given by
$$
	\Delta(R^{(n)}(x)):=\sum\limits_{i=0}^nR^{(i)}(x)\otimes R^{(n-i)}(x).
$$
Equivalently, the commutative Spitzer algebra is freely generated by the $R^{(n)}(x)$ and the coproduct defined by
$$
	\Delta(R(x^n)):=R(x^n)\otimes 1+1\otimes R(x^n).
$$
\end{defn}

The terminology ``Spitzer algebra'' was introduced in \cite{EGBP,EFMP} in the non-commutative setting, we extend it here to the commutative one.

It is interesting to analyse this Hopf algebra structure from the point of view of the Spitzer formula, which says that:
$$
	\log \big(1+\sum_{n>0} R^{(n)}(x)\big)=R\big(-\sum_{n>0}\frac{(-1)^nx^n}{n}\big).
$$
It is a general phenomenon in Hopf algebra, that the logarithm of a group-like element is primitive (and conversely, the exponential of a primitive element is group-like). It follows therefore from the Spitzer formula (without appealing to the theory of symmetric functions) that the two Hopf algebra structures agree. Moreover, Spitzer's formula gains now a group-theoretical flavour, i.e., $1+\sum_{n>0} R^{(n)}(x)$ belongs to the group of group-like elements of the commutative Spitzer algebra, whereas the $R(\frac{x^n}{n})$ belong to the associated Lie algebra. 

Here, the group is commutative and the Lie algebra structure trivial (the bracket is null). The interest of these observation is therefore limited. However, these phenomena will be more interesting in the non-commutative case, where they relate, for example, to the Baker--Campbell--Hausdorff and the Magnus formulas.

\smallskip

Let us return briefly to the first part of this article: the elements  $R^{(n)}(x)$ in the free Rota--Baxter algebra give rise to iterated integrals when $R$ is interpreted as an integral operator and $x$ as a smooth function. From Chen's rule we know that it would be natural to define a Hopf algebra structure in this context by requiring 
$$
	\Delta(R^{(n)}(x)):=\sum\limits_{i=0}^nR^{(i)}(x)\otimes R^{(n-i)}(x).
$$
This is indeed the structure we defined using Rota's embedding of symmetric functions into the free Rota--Baxter algebra. We will deepen this point of view and use it systematically later, when studying non-commutative Rota--Baxter algebras. 



\subsection{Spitzer's Identity: non-commutative case}
\label{ssect:ncSpitzer}

Thanks to the particular way we presented the commutative case of Spitzer's identity, we do not have to dwell too much on the details in the non-commutative setting. Instead, we will try to sketch the techniques as well as theoretical ingredients, that permit a simple transition from the commutative to the non-commutative Spitzer identity.   

\smallskip

The natural question is to understand the expression on the lefthand side of Spitzer's identity, i.e., the logarithm of  $\harpoonl{\exp}_{{\!\scriptscriptstyle{R}}}(\lambda x)$ in \eqref{atkinson}. Notice that, depending on the algebra $A$ and Rota--Baxter operator $R$ under consideration, one may either view this solution as the counterterm in renormalisation theory or the solution of a 1st order linear differential equation in matrix algebra, or the linear fixed point equation of a discrete dynamical system. In particular, in view of the penultimate example, one would expect that a ``non-commutative Spitzer formula'' should solve, among others, the Baker--Campbell--Hausdorff problem (to compute the logarithm of the solution of a 1st order linear differential equation). We will see that this is indeed the case.

Key in the theory of non-commutative Rota--Baxter algebras is the existence of another algebraic structure beside the Rota--Baxter product $\ast_\theta$. Indeed, in \eqref{postLie1}-\eqref{postLie2} and Proposition \ref{prop:pre-LieRB} we saw that any associative Rota--Baxter algebra of weight $\theta$ inherits both post-Lie and (left) pre-Lie algebra structures. 

To understand the role played by the pre-Lie product we return to the Bohnenblust--Spitzer formula, using the description presented in (\ref{clBSpPerm}), which appeared somewhat artificial in the commutative case. For $n=2$ we see quickly that:
$$
	F_1 \ast_\theta F_2 + \theta F_2F_1 = R(F_1)F_2  + R(F_2)F_1
$$
provided that the Rota--Baxter algebra is commutative. In the non-commutative case the slightly less obvious identity holds:  
\begin{eqnarray*}
	F_1 *_{\theta} F_2 + F_2 \bullet_\theta  F_1 = R(F_1)F_2  + R(F_2)F_1.				
\end{eqnarray*} 
This simple procedure of replacing the algebra product by the pre-Lie product, $\bullet_\theta$ defined in Proposition \ref{prop:pre-LieRB}, generalises to all orders. Indeed, it is ``sufficient" to redefine in (\ref{clBSpPerm}) the operator $\mathcal{D}^\theta_\sigma$ by substituting in definition~(\ref{lr}) the operator {\makebox{$r_{\theta x}$ (right product by $\theta x$)} with the operator {\makebox{$r_{{\bullet_\theta} x}$}}, again a right multiplication operator, but defined now in terms of the pre-Lie product, {\makebox{$r_{{\bullet_\theta} x}(y):=y \bullet_\theta x$}}. 

Using the example following the definition~(\ref{lr}), we substitute in the Bohnenblust--Spitzer formula the expression:
$$
	\mathcal{D}^\theta_{\sigma}(F_1,\ldots,F_5) = \theta^3 (F_4 F_3)*_{\theta}(F_5 F_1F_2),
$$  
by the one defined in terms of the pre-Lie product: 
\begin{eqnarray*}
	\mathcal{D}^{\bullet_\theta} _{\sigma}(F_1,\ldots,F_5) 
	&=&r_{\bullet_\theta  F_{3}}(F_{4}) *_\theta r_{{\bullet_\theta}  F_{2}}(r_{\bullet_\theta  F_{1}}(F_{5})) \\ 
	&=& (F_4 \bullet_\theta F_3)*_{\theta}((F_5\bullet_\theta  F_1)\bullet_\theta F_2).
\end{eqnarray*}

\begin{prop}
With this redefinition, the Bohnenblust--Spitzer identity (\ref{clBSpPerm}) holds in any non-commutative Rota--Baxter algebra of weight $\theta$ \cite{EFMP}:
\begin{equation}
	\sum_{\sigma \in S_n} R\bigl(\cdots R( R(F_{\sigma(1)})F_{\sigma(2)})\cdots \bigr)F_{\sigma(n)} 
	=	\sum_{\sigma \in S_n}  \mathcal{D}^{\bullet\theta}_{\sigma}(F_1,\ldots,F_n), 
\end{equation}  
where for each permutation $\sigma$ written in its canonical cycle decomposition $c_1 \cdots c_k$ we define:
\begin{equation}
	\mathcal{D}^{\bullet\theta}_{\sigma}(F_1,\ldots,F_n) 
	:= \sideset{}{^{*_{\theta}}}\prod_{j=1}^{k}\Big(\,\, \big(\,\, \sideset{}{^{\circ}}\prod_{i=1}^{m_j-1}r_{\bullet\theta F_{a_{j_i}}}\big)(F_{a_{j_0}}) \Big).
\end{equation}   
\end{prop}

Notice that the identity has even a $q$-analogue  \cite{NovThi}.

To approach the Baker--Campbell--Hausdorff problem generalised to Rota--Baxter algebra, it is useful to recall Spitzer's classical formula and its rewriting in terms of the generating series of the Bernoulli numbers given in (\ref{magnus}), and to compare it to Magnus classical solution of the problem. We already recalled that in his seminal 1954 paper \cite{Magnus}, Magnus considered the (say, matrix valued) solution $Y(t)$ of the linear differential equation $\dot{Y}(t)=A(t)Y(t)$, $Y(0)=1$ and showed that the logarithm $\Omega(A)(t):=\log(Y(t))$ is the solution of the differential equation:
$$
	\dot{\Omega}(A)=\frac{ad_{\Omega(A)}}{\mathrm{e}^{ad_{\Omega(A)}}-1}(A),
$$
where $ad$ is the ordinary Lie adjoint action ($ad_{x}(y)=[x,y]=xy-yx$). We remark that, since $\Omega(A)(0) = 0$, the fact that the indefinite Riemann integral is a weight zero Rota--Baxter map, implies that:
$$
	ad_{\Omega(A)}(A)=[\Omega(A),A]=  \dot{\Omega}(A)  \bullet_0 A
				     = \ell_{\dot{\Omega}(A)\bullet_0 }(A), 
$$
where $\ell_{x{\bullet_0} }(y):=x\bullet_0 y$. 

The generalisation to arbitrary Rota--Baxter algebras unifies Spitzer's classical formula with Magnus' expansion thanks to the pre-Lie product:

\begin{prop} \label{prop:ncSpitzer} 
Let $A$ be a Rota--Baxter algebra of weight $\theta$. If $\harpoonl{\exp}_{{\!\scriptscriptstyle{R}}}(\lambda x)$ is a solution in $A[[\lambda]]$ of the left fixpoint equation in \eqref{atkinson}, then the element $\Omega'(\lambda x)$ defined such that $R(\Omega'(\lambda x))=\log(\harpoonl{\exp}_{{\!\scriptscriptstyle{R}}}(\lambda x))$ satisfies \cite{EFM2}:
\begin{equation}
\label{pLMag}
	\Omega'(\lambda x)
	=\frac{\ell_{ -\Omega'(\lambda x){\bullet_\theta}}}{\mathrm{e}^{\ell_{ -\Omega'(\lambda x){\bullet_\theta} }}-1}(\lambda x)
	=\lambda x+\sum_{n>0}\frac{B_n}{n!}\ell_{-\Omega'(\lambda x){\bullet_\theta}}^n(\lambda x),
\end{equation}
where $\ell_{x{\bullet_\theta} }(y):=x\bullet_\theta y$. This series is called pre-Lie Magnus expansion.
\end{prop} 

Here, the prime notation shall remind the reader of Magnus' original differential equation. As a remark we mention that a similar approach applies to Fer's expansion \cite{EFM2}.  Fr\'ed\'eric Chapoton showed in \cite{Chap} that (\ref{pLMag}) is of significant interest in the context of the theory of Lie idempotents as well as the theory of operads. Let us make the first few terms of  $\Omega'(\lambda x)$ explicit:
\begin{eqnarray*}
 	\Omega'(\lambda x) &= & 
	\lambda x 
	+ \frac 12 \lambda^2 x\bullet_\theta x 
	+\frac 14 \lambda^3  (x \bullet_\theta x) \bullet_\theta x
 	+\frac 1{12} \lambda^3 x \bullet_\theta  (x\bullet_\theta x) + \\
	&&\qquad - \lambda^4 \frac 18  ((x \bullet_\theta  x) \bullet_\theta x ) \bullet_\theta  x 
			- \lambda^4 \frac{1}{24}\Big( (x \bullet_\theta  (x \bullet_\theta  x )) \bullet_\theta  x\\
	&&\qquad \qquad + x \bullet_\theta  ((x \bullet_\theta  x ) \bullet_\theta x) 
	+  (x \bullet_\theta x) \bullet_\theta  (x  \bullet_\theta x)  \Big) + \cdots.
\end{eqnarray*}
Observe that the pre-Lie identity implies a reduction of the number of terms at order four: 
$$
	 -\frac{1}{6} \bigl((x \bullet_\theta x)  \bullet_\theta  x \bigr)  \bullet_\theta x
	 - \frac{1}{12} x  \bullet_\theta \bigl((x  \bullet_\theta x)  \bullet_\theta  x \bigr).
$$ 


We conclude this section with an observation. The attentive reader may wonder about the role of the post-lie algebra structure shown in \eqref{postLie1}-\eqref{postLie2}. 
The post-Lie product $x \triangleright y := [R(x),y]$ induced on the Lie Rota--Baxter algebra $\mathfrak g$, permits to solve the Lie bracket flow equation
$$
	\dot{x}(t)= - x \triangleright x, \quad x(0)=x_0.
$$ 
Indeed, by general results of post-Lie algebra, one can show that 
$$
	x(t) =  \exp\big(- R(\chi(x_0t))\big) x_0 \exp\big(R(\chi(x_0t))\big).
$$
Here, $\chi(x_0t) \in \mathfrak g[[t]]$ is the post-Lie Magnus expansion. We refrain from giving more details and refer to \cite{EFM2018,M2019} and references therein, where the far more intricate notion of post-Lie Magnus expansion and related topics are reviewed in detail.


\subsection{Free Rota--Baxter algebras}
\label{ssect:ncCartier-Rota}

We conclude this article by the study of free Rota--Baxter algebras, the word problem (find a basis for these algebras), and insights on Hopf algebra structures in relation to the developments in the first part of this article.
The non-commutative analogue of Cartier's theory of free commutative Rota--Baxter algebras is intimately related to non-commutative analogues of shuffle and quasi-shuffle algebras. A natural way to represent these structures is in terms of graphical, i.e., combinatorial objects, such as planar rooted trees \cite{AgMo,EFG1}.

On the other hand, the description of the non-commutative analogue of Rota's classical solution to the word problem  is surprisingly straightforward \cite{EGBP,EFMP}. One simply replaces in Rota's original work as accounted for earlier in this article the algebra of polynomials in an infinite number variables, $k[x_1,\ldots,x_n,\ldots]$, by its associative analogue, i.e., the free associative algebra (or tensor algebra) over the alphabet  $X:=\{x_1,\ldots,x_n,\ldots\}$. Rota's results continue to hold in this non-commutative setting, that is, the operator $R(x_1,\ldots,x_n,\ldots):=(0,x_1,x_1+x_2,\ldots,x_1+\cdots+x_n,\ldots)$ still is a Rota--Baxter map, and the Rota--Baxter subalgebra generated by the sequence $x=(x_1,\ldots,x_n,\ldots)$ gives a presentation of the free Rota--Baxter algebra over one generator $x$.         

Besides the fact that this gives a simple answer to the word problem, it generalises Rota's seminal insight, unveiling the link between free commutative Rota--Baxter algebras and symmetric functions, to one of the fundamental notions in modern algebraic combinatorics, i.e., the theory of non-commutative symmetric functions -- the latter has been developed in the last 25 years by J.-Y.~Thibon and his collaborators (G.~Duchamp, F.~Hivert, J.-C.~Novelli \it inter alia\rm ) in a series of important articles from \cite{gelfand} to \cite{duc}.  

Indeed, recall that, if we calculate the terms on the left hand side of Spitzer's classical identity in the commutative case, i.e., $R^{(n)}(x)=R(R^{(n-1)}(x)x)$, using Rota's presentation of the free Rota--Baxter algebra on one generator, then we find a sequence of elementary symmetric functions $\sum_{0<i_1< \cdots <i_n<k}x_{i_1} \cdots x_{i_n}$ -- this is the basis for Rota's proof of the Spitzer formula, as a corollary of Waring's identity. In the non-commutative case the same observation holds, but this time non-commutative variables enter the picture. As a result, Florent~Hivert's theory of quasi-symmetric functions in non-commutative variables applies. This approach shares the same advantages with Rota's original work. Indeed, it allows to apply a whole range of results and techniques from the theory of non-commutative symmetric functions.

\medskip

Let us conclude by explaining briefly how these ideas connect to the notion of descent algebra introduced at the beginning of the article. Recall that the construction of the descent algebra was motivated by the combinatorial formulas for products of iterated integrals. It resulted that the descent algebra is a free associative algebra over generators corresponding to iterated integrals over standard simplices.
In the Rota--Baxter context, these ideas generalise as follows. First, as in the commutative case, one can show that the $R^{(n)}(x)$ are free generators of a subalgebra of $RB(x)$, the free Rota--Baxter algebra over $x$ \cite{EGBP,EFMP}. One can equip this algebra with a Hopf algebra structure (from the point of view of iterated integrals, the coproduct is induced by the Chen rule):

\begin{defn}
The Spitzer algebra is the cocommutative Hopf algebra that, as an associative algebra, is the free associative algebra generated by the $R^{(n)}(x)$ in $RB(x)$. The coproduct is given by
$$
	\Delta(R^{(n)}(x)):=\sum\limits_{i=0}^nR^{(i)}(x)\otimes R^{(n-i)}(x).
$$ 
Notice that the map $1_n\longmapsto R^{(n)}(x)$ induces an isomorphism between  the descent algebra and the Spitzer algebra (as Hopf algebras).

Equivalently, the Spitzer algebra is freely generated by the $R^{(n)}(x)$ and the coproduct defined by
$$
	\Delta(R(x^n)):=R(x^n)\otimes 1+1\otimes R(x^n).
$$
\end{defn}

This construction allows, for example, to interpret the pre-Lie Magnus expansion for Rota--Baxter algebras as surveyed in the previous section in group and Lie theoretical terms (since $\harpoonl{\exp}_{{\!\scriptscriptstyle{R}}}(\lambda x)$ is a group-like element in the Spitzer algebra, and its logarithm a primitive element).

Examples of applications of descent and free Lie algebra techniques to the study of Rota--Baxter algebras can be found in \cite{EFGP,EGBP,EFMP}.



\end{document}